\documentclass[a4paper,reqno,11pt]{amsart}
\usepackage{amssymb}
\usepackage{amstext}
\usepackage{amsmath}
\usepackage{amscd}
\usepackage{amsthm}
\usepackage{amsfonts}
\usepackage{eepic}
\usepackage{enumerate}
\usepackage{graphicx}
\usepackage{latexsym}
\usepackage{mathrsfs}
\usepackage{tikz}
\usepackage[all]{xy}
\input xy
\xyoption{all}
\usepackage{pstricks}
\usepackage{lscape}
\usepackage{comment}
\usepackage{color}

\numberwithin{equation}{section}
\newtheorem{theorem}{Theorem}[section]

\newtheorem{lemma}[theorem]{Lemma}
\newtheorem{proposition}[theorem]{Proposition}
\newtheorem{assumption}[theorem]{Assumption}
\newtheorem{defprop}[theorem]{Definition-Proposition}

\theoremstyle{definition}
\newtheorem{definition}[theorem]{Definition}
\newtheorem{remark}[theorem]{Remark}
\newtheorem{example}[theorem]{Example}
\newcommand{\add}{\operatorname{\mathsf{add}}\nolimits}

\newcommand{\End}{\operatorname{End}\nolimits}
\newcommand{\Ext}{\operatorname{Ext}\nolimits}

\newcommand{\Hom}{\operatorname{Hom}\nolimits}

\newcommand{\kD}{\operatorname{D}}
\newcommand{\Ker}{\operatorname{Ker}\nolimits}

\renewcommand{\mod}{\operatorname{\mathsf{mod}}\nolimits}
\newcommand{\Mod}{\operatorname{\mathsf{Mod}}\nolimits}
\newcommand{\rad}{\operatorname{rad}\nolimits}

\newcommand{\thick}{\operatorname{\mathsf{thick}}\nolimits}
\newcommand{\Sub}{\operatorname{\mathsf{Sub}}\nolimits}
\newcommand{\USub}{\operatorname{\underline{\mathsf{Sub}}}\nolimits}

\newcommand{\supp}{\operatorname{Supp}\nolimits}
\newcommand{\Deg}{\operatorname{deg}\nolimits}

\newcommand{\injdim}{\operatorname{inj.dim}\nolimits}

\newcommand{\proj}{\operatorname{proj}\nolimits}
\newcommand{\Max}{\operatorname{max}\nolimits}
\newcommand{\xto}[1]{\xrightarrow{#1}}

\newcommand{\mbf}[1]{\mbox{\boldmath${#1}$}}

\begin{document}
\title[Tilting and cluster tilting for preprojective algebras and Coxeter groups]{Tilting and cluster tilting for preprojective algebras and Coxeter groups}
\author[Y. Kimura]{Yuta Kimura}
\address{Graduate School of Mathematics, Nagoya University, Frocho, Chikusaku, Nagoya, 464-8602, Japan}
\email{m13025a@math.nagoya-u.ac.jp}
\date{\today}
\begin{abstract}
We study the stable category of the factor algebra of the preprojective algebra associated with an element $w$ of the Coxeter group of a quiver.
We show that there exists a silting object $M(\mbf{w})$ of this category associated with each reduced expression $\mbf{w}$ of $w$ and give a sufficient condition on $\mbf{w}$ such that $M(\mbf{w})$ is a tilting object.
In particular, the stable category is triangle equivalent to the derived category of the endomorphism algebra of $M(\mbf{w})$.
Moreover, we compare it with a triangle equivalence given by Amiot-Reiten-Todorov for a cluster category.
\end{abstract}
\maketitle 
\section{Introduction}\label{intro}
Recently, there are many studies on 2-Calabi-Yau triangulated categories and their
cluster tilting objects.
A well-studied class of $2$-Calabi-Yau triangulated categories is the stable categories of self-injective preprojective algebras  \cite{GLS06}.
This class was generalized by Buan-Iyama-Reiten-Scott  \cite{BIRSc}.
They constructed $2$-Calabi-Yau triangulated categories from a preprojective algebra $\Pi$ of a finite acyclic quiver $Q$ and an element $w$ of the Coxeter group $W_{Q}$ of $Q$.
They introduced a factor algebra $\Pi(w)$ of $\Pi$ and showed that the stable category $\USub\Pi(w)$ of a Frobenius category $\Sub\Pi(w)$ is a $2$-Calabi-Yau triangulated category.
It was also shown that  there exists a cluster tilting object in $\USub\Pi(w)$ associated with a reduced expression of $w$.
The category $\Sub\Pi(w)$ has been extensively studied by a number of authors \cite{AIRT, IR}.

Another well-studied class of $2$-Calabi-Yau triangulated categories is the cluster categories.
The cluster category of a finite dimensional hereditary algebra was introduced by \cite{BMRRT}, and generalized by Amiot \cite{Amiot09} for a finite dimensional algebra $A$ of global dimension at most two.
Amiot showed that the cluster category $\mathsf{C}(A)$ of $A$ has cluster tilting objects if $\mathsf{C}(A)$ is $\Hom$-finite.

A relationship between $2$-Calabi-Yau triangulated categories $\mathsf{C}(A)$ and $\USub\Pi(w)$ was studied by Amiot-Reiten-Todorov \cite{ART}.
For any element $w\in W_{Q}$ and a reduced expression $\mbf{w}$ of $w$, they constructed a finite dimensional algebra $A(\mbf{w})$ (see Section \ref{sectionrelation}) and they showed that there exists a triangle equivalence
\begin{align}\label{aireq}
\mathsf{C}(A(\mbf{w}))\simeq\underline{\Sub}\,\Pi(w).
\end{align}

The first aim of this paper is to introduce a graded analogue of an existence of cluster tilting objects of $\USub\Pi(w)$.
The orientation of $Q$ gives a natural grading on the preprojective algebra $\Pi$ of $Q$ and $\Pi(w)$.
We consider the stable category $\USub^{\mathbb{Z}}\Pi(w)$ of a Frobenius category $\Sub^{\mathbb{Z}}\Pi(w)$, which is a graded analogue of $\USub\Pi(w)$. 
In our previous work \cite{Kimura}, we constructed a tilting object in $\USub^{\mathbb{Z}}\Pi(w)$ for a special class of elements $w$ of $W_{Q}$ called $c$-sortable.
In this paper, we first prove that for any element $w$ of $W_{Q}$ and its reduced expression ${\mbf{w}}$,  we construct a silting object $M(\mbf{w})$ of $\USub^{\mathbb{Z}}\Pi(w)$.

\begin{theorem}[Theorem \ref{silting}]\label{intro1}
Let $w\in W_{Q}$.
For any reduced expression ${\mbf{w}}$ of $w$, there exists a silting object $M(\mbf{w})$ of $\USub^{\mathbb{Z}}\Pi(w)$.
\end{theorem}

Our silting object $M(\mbf{w})$ is, as an ungraded $\Pi(w)$-module, isomorphic to the above tilting object for $c$-sortable case.
Note that our $M(\mbf{w})$ is not a tilting object of $\USub^{\mathbb{Z}}\Pi(w)$ in general (see Example \ref{exnottilting}).
Our second result gives a sufficient condition on $\mbf{w}$, which is much weaker than $c$-sortability, such that $M(\mbf{w})$ is a tilting object of $\USub^{\mathbb{Z}}\Pi(w)$.
Therefore we have a triangle equivalence between $\USub^{\mathbb{Z}}\Pi(w)$ and the derived category of the endomorphism algebra of $M(\mbf{w})$.

\begin{theorem}[Theorem \ref{tilting}]\label{intro2}
Let $w\in W_{Q}$ and ${\mbf{w}}$ be a reduced expression of $w$.
If ${\mbf{w}}$ is $c$-ending on $Q_{0}$ or $c$-starting on $Q_{0}$ (see Definition \ref{PM}), then we have
\begin{itemize}
\item[(a)]
the object $M=M(\mbf{w})\in\USub^{\mathbb{Z}}\Pi(w)$ is a tilting object,
\item[(b)]
the global dimension of the endomorphism algebra $\underline{\End}^{\mathbb{Z}}_{\Pi(w)}(M)$ of $M$ in $\USub^{\mathbb{Z}}\Pi(w)$ is at most two, and
\item[(c)]
there exists a triangle equivalence
${\mathsf D}^{{\rm b}}(\mod\underline{\End}^{\mathbb{Z}}_{\Pi(w)}(M))\simeq\USub^{\mathbb{Z}}\Pi(w)$.
\end{itemize}
\end{theorem}

The third aim of this paper is to compare the equivalence obtained by a tilting object $M(\mbf{w})$ and the equivalence (\ref{aireq}).
We show that if the endomorphism algebra $\underline{\End}^{\mathbb{Z}}_{\Pi(w)}(M(\mbf{w}))$ of $M(\mbf{w})$  coincides with the algebra $A(\mbf{w})$,
then two equivalences commutes with canonical functors.

\begin{theorem}[Theorem \ref{mainthm}]\label{intro3}
Let $w\in W_{Q}$ and ${\mbf{w}}$ be a reduced expression of $w$.
If ${\mbf{w}}$ is $c$-ending on $\supp(w)$, then $\underline{\End}^{\mathbb{Z}}_{\Pi(w)}(M(\mbf{w}))=A(\mbf{w})$ holds and we have the following commutative diagram up to isomorphism of functors
$$\xymatrix{
{\mathsf D}^{{\rm b}}(\mod A(\mbf{w})) \ar[r]^{\simeq} \ar[d]^{\pi} & \underline{\Sub}^{\mathbb{Z}}\Pi(w) \ar[d]^{{\rm Forget}} \\
\mathsf{C}(A(\mbf{w})) \ar[r]^{\simeq} & \underline{\Sub}\,\Pi(w),
}$$
where $\pi$ is a canonical triangle functor.
\end{theorem}

Such a commutative diagram consisting of an equivalence between a derived category and the stable category of graded modules and an equivalence between cluster category and the stable category of ungraded modules often appears in representation theory e.g. \cite{AIR,IO}.

This paper is organized as follows.
In Section \ref{prelimi}, we give some notation which we use in this paper, introduce a natural grading on preprojective algebras and show some basic lemmas.
The category $\Sub^{\mathbb{Z}}\Pi(w)$ is defined and a Serre functor of $\USub^{\mathbb{Z}}\Pi(w)$ is studied in this section.
In Section \ref{CTsubcat}, we first study a more general triangulated category than $\USub^{\mathbb{Z}}\Pi(w)$, that is, $\Hom$-finite, Krull-Schmidt triangulated category $\mathcal{T}$ with a Serre functor and a cluster tilting subcategory.
After that, we apply the result of $\mathcal{T}$ to the category $\USub^{\mathbb{Z}}\Pi(w)$ in order to show Theorem \ref{intro1}.
In Section \ref{sectiontilting}, we show Theorem \ref{intro2}.
Some concrete examples are given in this section.
In Section \ref{sectionrelation}, we briefly introduce cluster categories and show Theorem \ref{intro3}.

In this paper, we denote by $K$ an algebraically closed field.
All categories are $K$-categories.
All algebras are $K$-algebras, and all graded algebras are $\mathbb{Z}$-graded $K$-algebras.
We always deal with left modules.
For an algebra $A$, we denote by $\Mod A$ (resp, $\mod\,A$, $\mathsf{fd}\,A$) the category of (resp, finitely generated, finite dimensional) $A$-modules.
For a graded algebra $A$, we denote by  $\Mod^{\mathbb{Z}}A$ (resp, $\mod^{\mathbb{Z}}A$, $\mathsf{fd}^{\mathbb{Z}}A$) the category of (resp, finitely generated, finite dimensional) $\mathbb{Z}$-graded $A$-modules with degree zero morphisms.
For graded $A$-modules $M, N$, we denote by $\Hom_{A}^{\mathbb{Z}}(M,N)$ the set of morphisms from $M$ to $N$ in $\Mod^{\mathbb{Z}}A$.
For an additive category $\mathcal{C}$ and $M\in \mathcal{C}$, we denote by $\add(M)$ the additive closure of $M$ in $\mathcal{C}$.
The composition of morphisms $f: X \to Y$ and $g: Y \to Z$ is denoted by $fg=g\circ f: X \to Z$.
For two algebras $A$ and $B$, we denote by $A\otimes B$ the tensor algebra of $A$ and $B$ over $K$.
For two arrows $\alpha, \beta$ of a quiver such that the target of $\alpha$ is the source of $\beta$, we denote by $\alpha\beta$ the composition of $\alpha$ and $\beta$.
We denote by $\kD=\Hom_K(-,K)$ the standard $K$-dual.
\section{Preliminary}\label{prelimi}
In this section, we define some notation and show some lemmas which we use later.
Throughout this section, let $Q=(Q_{0},Q_{1},s,t)$ be a finite acyclic quiver, where for an arrow $\alpha$ of $Q$, we denote by $s(\alpha)$ the source of $\alpha$, and by $t(\alpha)$ the target of $\alpha$.
\subsection{Preprojective algebras and Coxeter groups}
The {\it double quiver} $\overline{Q}=(\overline{Q}_0,\overline{Q}_1,s,t)$ of the quiver $Q$ is defined by $\overline{Q}_0=Q_0$, $\overline{Q}_1=Q_1\sqcup \{\, \alpha^{\ast}: t(\alpha)\to s(\alpha) \mid \alpha \in Q_1 \,\}$.
The {\it preprojective algebra} $\Pi$ of $Q$ is definite by the following
\begin{align*}
	\Pi :=K\overline{Q}/\langle \displaystyle\sum\limits_{\alpha\in Q_1} \alpha\alpha^{\ast}-\alpha^{\ast}\alpha \rangle . 
	\end{align*}
The {\it Coxeter group} $W_{Q}$ of $Q$ is the group generated by the set  $\{\,s_{u} \mid u\in Q_0\,\}$ with relations $s_{u}^2=1$, $s_{u}s_{v}=s_{u}s_{v}$ if there exist no arrows between $u$ and $v$, and $s_us_vs_u=s_vs_us_v$ if there exists exactly one arrow between $u$ and $v$.
\begin{definition}\label{defcox}
Let $w\in W_{Q}$ and ${\mbf{w}}=s_{u_{1}}s_{u_{2}}\cdots s_{u_{l}}$ be an expression of $w$.
\begin{itemize}
	\item[(1)] A {\it subword} of ${\mbf{w}}$ is an expression ${\mbf{w}}^{\prime}=s_{u_{i_{1}}}s_{u_{i_{2}}}\cdots s_{u_{i_{m}}}$ such that $1\leq i_{1} < i_{2} < \cdots <i_{m}\leq l$.
	
	\item[(2)] An expression ${\mbf{w}}$ of $w$ is {\it reduced} if $l$ is smallest possible.
	
	\item[(3)] Let ${\mbf{w}}$ be a reduced expression of $w$, put $\supp(w):=\{ u_{1}, u_{2}, \ldots, u_{l} \}\subset Q_{0}$. Note that, $\supp(w)$ is independent of the choice of a reduced expression of $w$.
	
	\item[(4)] An element $c\in W_{Q}$ is called a {\it Coxeter element} if there exists an expression $s_{v_{1}}s_{v_{2}}\cdots s_{v_{n}}$ of $c$ such that $\{ v_{1}, v_{2}, \ldots , v_{n} \}$ is a permutation of $Q_{0}$ and $e_{v_j}KQe_{v_i}=0$ holds for $i < j$.
\end{itemize}
\end{definition}
\if()
\begin{definition}
Let $c$ be the Coxeter element of $W$. An element $w\in W$ is called a {\it $c$-sortable element} if there exists a reduced expression of $w$ of the form $w=c^{(0)}c^{(1)}\cdots c^{(m)}$, where each $c^{(i)}$ is subsequence of $c$ and
	\begin{align}
	\supp(c^{(m)}) \subset \supp(c^{(m-1)}) \subset \cdots \subset \supp(c^{(0)}) \subset Q_0. \notag
	\end{align}
\end{definition}
\fi
Note that, since $Q$ is finite and acyclic, a Coxeter element $c$ of $W_{Q}$ exists and it is unique as an element of $W_{Q}$.
Therefore we call $c$ the Coxeter element.

Let $u$ be a vertex of $Q$.
We define a two-sided ideal $I_{u}$ of $\Pi$ by \[I_{u}:=\Pi(1-e_{u})\Pi.\]
Let $w$ be an element of $W_{Q}$ and ${\mbf{w}}=s_{u_{1}}s_{u_{2}}\cdots s_{u_{l}}$ be a reduced expression of $w$. We define a two-sided ideal $I(w)$ of $\Pi$ by \[I(w):=I_{u_{1}}I_{u_{2}}\cdots I_{u_{l}}.\]
Note that $I(w)$ is independent of the choice of a reduced expression $\mbf{w}$ of $w$ by  \cite[Theorem I\hspace{-.1em}I\hspace{-.1em}I. 1.9]{BIRSc}.
We define an algebra $\Pi(w)$ by \[\Pi(w):=\Pi/I(w).\]
For an algebra $A$, we denote by $\Sub A$ the full subcategory of $\mod A$ of submodules of finitely generated free $A$-modules.
A finite dimensional algebra $A$ is said to be {\it Iwanaga-Gorenstein of dimension at most one} if $\injdim {}_{A}A\leq 1$ and $\injdim A_{A}\leq 1$ hold.
We recall the following results.
\begin{proposition}\cite{BIRSc}\label{birs}
For any $w\in W_{Q}$, we have the following.
\begin{itemize}
\item[(a)] The algebra $\Pi(w)$ is finite dimensional and Iwanaga-Gorenstein of dimension at most one.
\item[(b)] $\Sub\Pi(w)$ is a Frobenius category, and the stable category $\USub \Pi(w)$ is a $2$-Calabi-Yau triangulated category, that is, for any objects $X,Y\in\USub\Pi(w)$, there exists a bifunctorial isomorphism $\underline{\Hom}_{\Pi(w)}(X,Y)\simeq\kD\underline{\Hom}_{\Pi(w)}(Y,X[2])$.
\item[(c)] For any reduced expression ${{\mbf{w}}}=s_{u_1}s_{u_2}\cdots s_{u_l}$ of $w$, the object $T=\bigoplus_{i=1}^{l}\Pi(s_{u_{1}}s_{u_{2}}\cdots s_{u_{i}})$ is a cluster tilting object of $\Sub\Pi(w)$, that is, $\add T=\{\,X\in\mod \Pi(w)\mid \Ext_{\Pi(w)}^{1}(X,T)=0\,\}$.
\end{itemize}
\end{proposition}
\subsection{Graded algebras}
In this subsection, we observe some properties of graded algebras.

A graded algebra $A=\bigoplus_{i\in\mathbb{Z}}A_{i}$ is said to be {\it positively graded} if $A_{i}=0$ for any $i< 0$.
Let $X=\bigoplus_{i\in\mathbb{Z}}X_i$ be a graded module  over a positively graded algebra.
For any integer $j$, we define the graded module $X(j)$ by $(X(j))_i=X_{i+j}$.
Moreover, for any integer $j$, we define a graded submodule $X_{\geq j}$ of $X$ by
	\begin{align}
	(X_{\geq j})_i=\begin{cases}
					X_i & \text{$i\geq j$} \\
					0 & \text{else}
					\end{cases} \notag
	\end{align}
and define a graded factor module $X_{\leq j}$ of $X$ by $X_{\leq j}=X/(X_{\geq {j+1}})$.
For a positively graded algebra $A$, we denote by $\mod^{\geq i}A$ the full subcategory of $\mod^{\mathbb{Z}}A$ whose objects satisfy $M=M_{\geq i}$.
We show the following two lemmas.
\begin{lemma}\label{ungradediso}
Let $A$ be a finite dimensional graded algebra and let $M,N$ be finitely generated indecomposable graded $A$-modules.
If $M$ is isomorphic to $N$ in $\mod\,A$, then there exists an integer $i$ such that $M(i)$ is isomorphic to $N$ in $\mod^{\mathbb{Z}}A$.
\end{lemma}
\begin{proof}
Let $f : M\to N$ be an isomorphism and $g : N \to M$ an inverse of $f$ in $\mod\,A$.
Since $M, N$ are finitely generated, we have $f=\sum_{s=1}^{k}f_{i_{s}}$, and $g=\sum_{t=1}^{l}g_{j_{t}}$, where $f_{i_{s}}\in\Hom_{A}^{\mathbb{Z}}(M,N(i_{s}))$ and $g_{j_{t}}\in\Hom_{A}^{\mathbb{Z}}(N,M(j_{t}))$.
We have the following morphisms in $\mod^{\mathbb{Z}}A$
\begin{align*}
\tilde{f}=\left(\begin{array}{c}f_{i_{1}} \\ \vdots\\ f_{i_{k}}\end{array}\right)
: M \to \bigoplus_{s=1}^{k}N(i_{k}), \quad \tilde{g}=(g_{-i_{1}}, \cdots, g_{-i_{k}}): \bigoplus_{s=1}^{k}N(i_{k}) \to M,
\end{align*}
where $g_{-i_{s}}=0$ if $-i_{s}\notin\{j_{1},\cdots, j_{l}\}$.
Then $\tilde{f}\tilde{g}=id_{M}$ holds, since $fg=id_{M}$.
Therefore $M$ is a direct summand of $\bigoplus_{s=1}^{k}N(i_{k})$.
Since $M$ is indecomposable, $M$ is isomorphic to $N(i)$ for some $i$.
\end{proof}
\begin{lemma}\label{A_0}
Let $A$ be a finite dimensional positively graded algebra such that the global dimension of $A_{0}$ is at most $m$.
Let $M\in\mod^{\geq 0}A$ and
\begin{align}\label{resol}
\cdots\to P^{2} \xto{f^{2}} P^{1} \xto{f^{1}} P^{0} \xto{f^{0}} M \to 0
\end{align}
be a minimal projective resolution of $M$ in $\mod^{\mathbb{Z}}A$.
Then we have $\Ker(f^{m})_{0}\in\mod^{\geq -1}A$.
\end{lemma}
\begin{proof}
By taking the degree zero part of (\ref{resol}), we have a minimal projective resolution of $M_{0}$ in $\mod\,A_{0}$.
Therefore we have $\Ker(f^{m})_{0}=\Ker(f^{m}|_{P^{n}_{0}})=0$.
\end{proof}
We use the following definition in Section \ref{sectionrelation}.
\begin{definition}\label{gradontensor}
Let $A,B,C$ be graded algebras.
\begin{itemize}
	\item[(1)]
	We define a grading on the tensor algebra $A\otimes B$ as follows:
	\begin{align*}
	\left(A\otimes B\right)_{i}=\left\{\, \sum a\otimes b \mid a \in A_{j},\, b\in B_{k},\, j+k=i \,\right\},
	\end{align*}
	for any $i\in\mathbb{Z}$.
	\item[(2)]
	Let $X$ be a graded $A\otimes B^{\rm op}$-module and $Y$ a graded $B\otimes C^{\rm op}$-module.
	We define a grading on the $A\otimes C^{\rm op}$-module $X\otimes_{B}Y$ as follows:
	\begin{align*}
	\left(X\otimes_{B}Y\right)_{i}=\left\{\, \sum x\otimes y \mid x\in X_{j},\, y\in Y_{k},\, j+k=i  \,\right\},
\end{align*}
	for any $i\in\mathbb{Z}$.
\end{itemize}
\end{definition}
\subsection{Grading on preprojective algebras}
In this subsection, we introduce a grading on preprojective algebras and observe some properties of two sided ideals $I(w)$ from the context of a grading on preprojective algebras.

We define a map $\deg : \overline{Q}_{1}\to \{0,1\}$ as follows: for each $\beta\in \overline{Q}_{1}$, let
	\begin{align}
	\deg (\beta) = \begin{cases}
				1 & \text{$\beta = \alpha^{\ast},\,\alpha \in Q_1$,} \\
				0 & \text{$\beta = \alpha ,\, \alpha \in Q_1.$}
			\end{cases}\notag
	\end{align}
We regard the path algebra $K\overline{Q}$ as a graded algebra by the map $\deg$.
Since the element $\sum\limits_{\alpha\in Q_1}(\alpha\alpha^{\ast}-\alpha^{\ast}\alpha)$ in $K\overline{Q}$ is homogeneous of degree $1$, the grading of $K\overline{Q}$ naturally gives a grading on the preprojective algebra $\Pi=\bigoplus\limits_{i\geq 0}\Pi_{i}$.
Preprojective algebras are positively graded with respect to the above grading.
\begin{remark}
\begin{itemize}
\item[(a)] We have $\Pi_0=KQ$, since $\Pi_0$ is spanned by all paths of degree $0$.
\item[(b)] For any $w\in W_{Q}$, the ideal $I(w)$ of $\Pi$ is a homogeneous ideal of $\Pi$ since so is each $I_{u}$.
\item[(c)] In particular, the factor algebra $\Pi(w)$ is a graded algebra.
\end{itemize}
\end{remark}
For a graded algebra $A$, let $\Sub^{\mathbb{Z}}A$ be the full subcategory of $\mod^{\mathbb{Z}}A$ of submodules of graded free $A$-modules, that is,
 	\begin{align}
	\Sub^{\mathbb{Z}}A=\biggl\{\, X\in\mod^{\mathbb{Z}}A\mid X \hspace{0.2cm} \textnormal{is a submodule of} \hspace{0.2cm} \bigoplus\limits_{i=1}^{m}A(j_{i}),\,\, m,j_i\in\mathbb{Z}, m\geq 0 \,\biggr\}. \notag
	\end{align}
Since $\Pi(w)$ is Iwanaga-Gorenstein of dimension at most one, $\Sub^{\mathbb{Z}}\Pi(w)$ is a Frobenius category, where the projective-injective objects of $\Sub^{\mathbb{Z}}\Pi(w)$ is the graded projective $\Pi(w)$-modules.
For any $X,Y\in\Sub^{\mathbb{Z}}\Pi(w)$,
we denote by $\mathcal{P}(X,Y)$ a subspace of $\Hom_{\Pi(w)}^{\mathbb{Z}}(X,Y)$ of morphisms factoring through a graded projective $\Pi(w)$-module and let
\begin{align*}
\underline{\Hom}_{\Pi(w)}^{\mathbb{Z}}(X,Y)=\Hom_{\Pi(w)}^{\mathbb{Z}}(X,Y)/\mathcal{P}(X,Y).
\end{align*}
Then we have a triangulated category $\USub^{\mathbb{Z}}\Pi(w)$, whose objects are the same as  $\Sub^{\mathbb{Z}}\Pi(w)$ and the morphism space from $X$ to $Y$ is $\underline{\Hom}_{\Pi(w)}^{\mathbb{Z}}(X,Y)$.

We need the following two lemmas.
\begin{lemma}\cite[Lemma 2.1]{AIRT}\label{airtlem2.1}
Let $Q^{\prime}$ be a full subquiver of $Q$ and  $w$ an element of $W_{Q^{\prime}}$.
Then we have $\Pi/I(w) = \Pi^{\prime}/I^{\prime}(w)$ as graded algebras,
where $\Pi^{\prime}$ is a preprojective algebra of $Q^{\prime}$ and $I^{\prime}(w)$ is the ideal of $\Pi^{\prime}$ associated with $w$.
\end{lemma}
\begin{lemma}\label{euev}
The following holds.
\begin{itemize}
\item[(a)] Let $c\in W_{Q}$ be the Coxeter element. We have $I(c)_{0}=0$.
\item[(b)] Let $w$ be an element of $W_{Q}$.
If there exists a reduced expression ${\mbf{w}}$ of $w$ containing an expression of the Coxeter element $c$ as a subword, then we have $I(w)_{0}=0$. In particular, we have $\Pi(w)_{0}=KQ$.
\item[(c)] Let ${\mbf{w}}=s_{u_{1}}s_{u_{2}}\cdots s_{u_{l}}$ be a reduced expression of $w\in W_{Q}$ which is a subword of an expression of the Coxeter element $c$ of $W_{Q}$.
Then we have $e_{u_{1}}I(w)_{0}e_{u_{l}}=0$.
\end{itemize}
\end{lemma}
\begin{proof}
(\rm a)
This comes from \cite[Proposition I\hspace{-.1em}I\hspace{-.1em}I. 3.2]{BIRSc}.

(\rm b)
By (\rm a), we have $I(w)_{0}\subset I(c)_{0}=0$.

(\rm c)
We have $e_{u_{1}}I(w)_{0}e_{u_{l}}\subset e_{u_{1}}I(c)_{0}e_{u_{l}}=0$.
\end{proof}
\subsection{Basic results on $\USub^{\mathbb{Z}}\Pi(w)$}\label{serrepiw}
Let $w\in W_{Q}$.
In this subsection, we show that the category $\USub^{\mathbb{Z}}\Pi(w)$ has a Serre functor and recall one theorem on $\USub^{\mathbb{Z}}\Pi(w)$.

We call a category $\mathcal{C}$ {\it $\Hom$-finite} if the $K$-vector space $\Hom_{\mathcal{C}}(X,Y)$ is finite dimensional for any $X,Y\in \mathcal{C}$.
For a $\Hom$-finite category $\mathcal{C}$, a {\it Serre functor} $\mathbb{S}$ is an auto-equivalence of $\mathcal{C}$ such that there exists a bifunctorial isomorphism $\Hom_{\mathcal{C}}(X,Y)\simeq\kD\Hom_{\mathcal{C}}(Y,\mathbb{S}(X))$ for any $X,Y\in\mathcal{C}$.

We need the following two observations.
\begin{lemma}\label{extclosed}
The category $\Sub^{\mathbb{Z}}\Pi(w)$ is an extension closed subcategory of $\mathsf{fd}^{\mathbb{Z}}\Pi$.
\end{lemma}
\begin{proof}
By \cite[Proposition I\hspace{-.1em}I\hspace{-.1em}I. 2.3]{BIRSc} (a), $\Sub\Pi(w)$ is an extension closed subcategory of $\mathsf{fd}\,\Pi$.
Let $0\to X\to Y\to Z\to 0$ be an exact sequence in $\mathsf{fd}^{\mathbb{Z}}\Pi$ and $X,Z\in\Sub^{\mathbb{Z}}\Pi(w)$.
Since $Y\in\Sub\Pi(w)$ and $Y\in\mathsf{fd}^{\mathbb{Z}}\Pi$, we have $Y\in\Sub^{\mathbb{Z}}\Pi(w)$.
\end{proof}
\begin{proposition}\label{serred}
Let $Q$ be a non-Dynkin quiver and $\Pi$ be the preprojective algebra of $Q$.
Put $\Pi^{e}=\Pi\otimes\Pi^{\rm op}$.
Let $\mathcal{D}={\mathsf D}(\Mod^{\mathbb{Z}}\Pi)$ be the derived category of $\Mod^{\mathbb{Z}}\Pi$ and $X, Y$ in $\mathcal{D}$.
Then the following holds.
\begin{itemize}
\item[(a)]
$R\Hom_{\Pi^{e}}(\Pi,\Pi^{e})\simeq \Pi[-2](1)$ holds in ${\mathsf D}(\Mod^{\mathbb{Z}}\Pi^{e})$.
\item[(b)]
If the homology of $X$ is of finite total dimension, then we have a bifunctorial isomorphism
\begin{align*}
\Hom_{\mathcal{D}}(X,Y)\simeq \kD\Hom_{\mathcal{D}}(Y,X[2](-1)).
\end{align*}
\end{itemize}
\end{proposition}
\begin{proof}
({\rm a})
By \cite[Section 8]{GLS07}, we have a graded $\Pi^{e}$-module resolution of $\Pi$:
	\begin{align*}
	0\to
	\bigoplus_{u\in Q_{0}}\left( \Pi e_{u}\otimes e_{u}\Pi \right)(-1) &
	\to
	\bigoplus_{\beta\in \overline{Q}_{1}}\left( \Pi e_{s(\beta)}\otimes e_{t(\beta)}\Pi \right)(-\Deg\beta)\\
	&\to
	\bigoplus_{u\in Q_{0}}\left( \Pi e_{u}\otimes e_{u}\Pi \right)
	\to \Pi \to 0.
	\end{align*}
This resolution gives us the desired isomorphism.

({\rm b})
This follows from ({\rm a}) and \cite[Lemma 4.1]{Ke08}.
\end{proof}
Then we have a Serre functor of $\USub^{\mathbb{Z}}\Pi(w)$.
\begin{proposition}\label{serresub}
For any $w\in W_{Q}$, the triangulated category $\USub^{\mathbb{Z}}\Pi(w)$ has a Serre functor $[2](-1)$.
\end{proposition}
\begin{proof}
By Lemma \ref{airtlem2.1}, we assume that $Q$ is a non-Dynkin quiver.
By Lemma \ref{extclosed}, $\Sub^{\mathbb{Z}}\Pi(w)$ is an extension closed full subcategory in $\mathsf{fd}^{\mathbb{Z}}\Pi$.
Thus we have $\Ext_{\Sub^{\mathbb{Z}}\Pi(w)}^{1}(X,Y)=\Ext_{\Mod^{\mathbb{Z}}\Pi}^{1}(X,Y)$ for $X,Y\in \Sub^{\mathbb{Z}}\Pi(w)$.
Therefore we have 
	\begin{align*}
	\Hom_{\USub^{\mathbb{Z}}\Pi(w)}(X,Y[1]) & \simeq \Ext_{\Mod^{\mathbb{Z}}\Pi}^{1}(X,Y)\\
	& \simeq \kD\Ext_{\Mod^{\mathbb{Z}}\Pi}^{1}(Y,X(-1))\\
	& \simeq \kD\Hom_{\USub^{\mathbb{Z}}\Pi(w)}(Y,X[1](-1)),
	\end{align*}
for $X,Y$ in $\Sub^{\mathbb{Z}}\Pi(w)$, where the second isomorphism comes from Proposition \ref{serred}.
 This means that $\underline{\Sub}^{\mathbb{Z}}\Pi(w)$ has a Serre functor $[2](-1)$.
\end{proof}
We need one result of Iwanaga-Gorenstein algebras.
The next theorem is the famous result of \cite{B, H, R} and its graded version in the case of injective dimension at most one.
For a finite dimensional (resp, graded) algebra $A$, we denote by ${\mathsf K}^{{\rm b}}(\proj A)$ (resp, ${\mathsf K}^{{\rm b}}(\proj^{\mathbb{Z}} A)$) the homotopy category of bounded complexes of finitely generated (resp, graded) projective $A$-modules.
\begin{theorem}
Let $A$ be an Iwanaga-Gorenstein algebra of dimension at most one.
Then the following holds.
\begin{itemize}
	\item[(a)]
	There exists a triangle equivalence
	\begin{align*}
	{\mathsf D}^{{\rm b}}(\mod A)/{\mathsf K}^{{\rm b}}(\proj A) \xto{\sim} \USub A,
	\end{align*}
	where a quasi-inverse of this equivalence is induced from the composite of the canonical functors $\Sub A \to {\mathsf D}^{{\rm b}}(\mod A) \to {\mathsf D}^{{\rm b}}(\mod A)/{\mathsf K}^{{\rm b}}(\proj A)$.
	
	\item[(b)]
	If $A$ is a graded algebra.
	Then we have the following triangle equivalence
		\begin{align*}
	{\mathsf D}^{{\rm b}}(\mod^{\mathbb{Z}} A)/{\mathsf K}^{{\rm b}}(\proj^{\mathbb{Z}} A) \xto{\sim} \underline{\Sub}^{\mathbb{Z}}A.
	\end{align*}
	where a quasi-inverse of this equivalence is induced from the composite of the canonical functors $\Sub^{\mathbb{Z}} A \to {\mathsf D}^{{\rm b}}(\mod^{\mathbb{Z}} A) \to {\mathsf D}^{{\rm b}}(\mod^{\mathbb{Z}} A)/{\mathsf K}^{{\rm b}}(\proj^{\mathbb{Z}} A)$.
\end{itemize}
\end{theorem}
Note that categories $\USub A$ and $\USub^{\mathbb{Z}}A$ for an Iwanaga-Gorenstein algebra $A$ of dimension at most one are often called singularity categories.
We denote by $\rho_{A}$ the composite of triangle functors
\begin{align*}
\rho_{A}: {\mathsf D}^{{\rm b}}(\mod\,A) \to {\mathsf D}^{{\rm b}}(\mod\,A)/{\mathsf K}^{{\rm b}}(\proj\,A) \xto{\sim} \underline{\Sub}A,
\end{align*}
and denote by $\rho_{A}^{\mathbb{Z}}$ the graded version of $\rho_{A}$ if $A$ is a graded algebra.
\subsection{Silting and tilting objects of triangulated categories}\label{subsectionsiltingtilting}
In this subsection, we recall the definition of silting and tilting objects and tilting theorem for triangulated categories which was shown by Keller.

Let $\mathcal{T}$ be a triangulated category.
For an object $X$ of $\mathcal{T}$, we denote by  $\thick X$ the smallest triangulated full subcategory of $\mathcal{T}$ containing $X$ and closed under direct summands.
\begin{definition}\label{defofsilting}
Let $\mathcal{T}$ be a triangulated category.
\begin{itemize}
\item[(1)]
An object $X$ of $\mathcal{T}$ is called a {\it silting object} if $\Hom_{\mathcal{T}}(X,X[i])=0$ for any $0<i$ and $\thick X=\mathcal{T}$.
\item[(2)]
An object $X$ of $\mathcal{T}$ is called a {\it tilting object} if $X$ is a silting object of $\mathcal{T}$ and $\Hom_{\mathcal{T}}(X,X[i])=0$ for any $i< 0$.
\end{itemize}
\end{definition}
For example, let $A$ be a finite dimensional algebra.
Then $A$ is a tilting object of ${\mathsf K}^{{\rm b}}(\proj A)$.

The following lemma is a fundamental observation for triangle functors and tilting objects.
An additive category $\mathcal{C}$ is called {\it Krull-Schmidt} if each object of $\mathcal{C}$ is a finite direct sum of objects such that whose endomorphism algebras are local.
\begin{lemma}\label{trifunclem}
Let $\mathcal{T,U}$ be triangulated categories and $F: \mathcal{T}\to\mathcal{U}$ be a triangle functor. 
Moreover, let $X$ be a tilting object of $\mathcal{T}$.
Assume that $\mathcal{T}$ is Krull-Schmidt and $F(X)$ is a tilting object of $\mathcal{U}$.
Then the functor $F$ is an equivalence.
\end{lemma}
We recall the following theorem shown by Keller.
\begin{theorem}\cite[(4.3)]{Ke94}\label{keller}
Let $\mathcal{T}$ be the stable category of a Frobenius category, 
and assume that $\mathcal{T}$ is Krull-Schmidt.
If there exists a tilting object $X$ of $\mathcal{T}$, then there exists a triangle equivalence
$\mathcal{T} \simeq {\mathsf K}^{{\rm b}}(\proj \End_{\mathcal{T}}(X))$.
\end{theorem}
\section{A silting object in $\USub^{\mathbb{Z}}\Pi(w)$}\label{CTsubcat}
In this section, we show that the category $\USub^{\mathbb{Z}}\Pi(w)$ has a silting object for any $w\in W_{Q}$.
In subsection \ref{subCTsubcat}, we study a more general triangulated category than $\USub^{\mathbb{Z}}\Pi(w)$.
\subsection{Cluster tilting subcategories and thick subcategories}\label{subCTsubcat}
In this subsection, let $\mathcal{T}$ be a Hom-finite, Krull-Schmidt triangulated category with a Serre functor $\mathbb{S}$.
Put $\mathbb{S}_{2}=\mathbb{S}\circ[-2]$.
We denote by $\mathcal{T}/\mathbb{S}_{2}$ the orbit category of $\mathcal{T}$ associated with $\mathbb{S}_{2}$.
For any object $M$ of $\mathcal{T}$, we regard the endomorphism algebra $\End_{\mathcal{T}/\mathbb{S}_{2}}(M)$ as a graded algebra by $\End_{\mathcal{T}/\mathbb{S}_{2}}(M)_{i}=\Hom_{\mathcal{T}}(M,\mathbb{S}_{2}^{-i}(M))$.
For a subcategory $\mathcal{C}$ of $\mathcal{T}$,
put $\mathcal{C}^{\bot}=\{X\in\mathcal{T} \mid \Hom_{\mathcal{T}}(\mathcal{C},X)=0\}$ and 
$^{\bot}\mathcal{C}=\{X\in\mathcal{T} \mid \Hom_{\mathcal{T}}(X,\mathcal{C})=0\}$.

A subcategory $\mathcal{C}$ of $\mathcal{T}$ is called a {\it contravariantly finite subcategory} of $\mathcal{T}$ if for any $X\in\mathcal{T}$,
there exists a morphism $f : Y\to X$ with $Y\in\mathcal{C}$ such that the map $\Hom_{\mathcal{T}}(Z,f) : \Hom_{\mathcal{T}}(Z,Y)\to\Hom_{\mathcal{T}}(Z,X)$ is surjective for any $Z\in\mathcal{C}$.
Dually, we define a {\it covariantly finite subcategory} of $\mathcal{T}$.
We call $\mathcal{C}$ a {\it functorially finite subcategory} of $\mathcal{T}$ if $\mathcal{C}$ is a contravariantly and covariantly finite subcategory of $\mathcal{T}$.

We recall the definition of cluster tilting subcategories.
\begin{definition}\cite{IY}
Let $\mathcal{C}$ be a subcategory of $\mathcal{T}$.
We call $\mathcal{C}$ a {\it cluster tilting subcategory} of $\mathcal{T}$ if $\mathcal{C}$ is a functorially finite subcategory of $\mathcal{T}$ and
\begin{align*}
\mathcal{C}=\mathcal{C}[-1]^{\bot}={}^{\bot}\mathcal{C}[1].
\end{align*}
\end{definition}
We recall the following property of cluster tilting subcategories.
\begin{proposition}\cite[Theorem 3.1]{IY}\label{iythm3.1}
If $\mathcal{C}$ is a cluster tilting subcategory of $\mathcal{T}$, then for any object $X$ of $\mathcal{T}$, there exists a triangle $C_{0} \to X \to C_{1}[1] \to C_{0}[1]$ with $C_{0},C_{1}\in\mathcal{C}$.
\end{proposition}
We recall some definitions.
We denote by $J_{\mathcal{T}}$ the {\it Jacobson radical} of $\mathcal{T}$.
We call a morphism $f:X\to Y$ in $\mathcal{T}$ {\it right minimal} if $f$ does not have a direct summand of the form $X^{\prime} \to 0$ for some $X^{\prime}\in\mathcal{T}$.
Let $\mathcal{C}$ be a full subcategory of $\mathcal{T}$.
A morphism $f:X\to Y$ in $\mathcal{C}$ is called a {\it right minimal almost split morphism} of $Y$ in $\mathcal{C}$ if the following three conditions are satisfied:
\begin{itemize}
\item[(i)] $f$ is not a retraction.
\item[(ii)] 
$f$ induces a surjective map $\Hom_{\mathcal{T}}(Z,X)\to J_{\mathcal{T}}(Z,Y)$ for any $Z\in \mathcal{C}$.
\item[(iii)] $f$ is right minimal.
\end{itemize}
Dually, a {\it left minimal almost split morphism} is defined.

Note that if there exists a left (resp, right) minimal almost split morphism of $Y$ in $\mathcal{C}$, then it is unique up to isomorphism.
We use the following theorem.
\begin{theorem}\cite[Theorem 3.10]{IY}\label{iythm3.10}
Let $\mathcal{C}$ be a cluster tilting subcategory of $\mathcal{T}$ and $X$ be an indecomposable object of $\mathcal{C}$.
Then there exist triangles
\begin{align}\label{arend}
\mathbb{S}_{2}(X) \xto{g} C_{1} \to Y \to \mathbb{S}_{2}(X)[1], \quad Y \to C_{0} \xto{f} X \to Y[1],
\end{align}
where $f$ is a right minimal almost split morphism in $\mathcal{C}$ and $g$ is a left minimal almost split morphism in $\mathcal{C}$.
Dually, there exist triangles
\begin{align}\label{arst}
X \xto{g^{\prime}} C^{0} \to Z \to X[1], \quad Z \to C^{1} \xto{f^{\prime}} \mathbb{S}_{2}^{-1}(X) \to Z[1],
\end{align}
where $g^{\prime}$ is a left minimal almost split in $\mathcal{C}$ and $f^{\prime}$ is a right minimal almost split in $\mathcal{C}$.
\end{theorem}
Note that the triangles (\ref{arst}) are obtained by applying the functor $\mathbb{S}_{2}^{-1}$ to the triangles (\ref{arend}).
In \cite{IY}, the triangles (\ref{arend}), regarded as a complex of $\mathcal{T}$, is called an {\it Auslander-Reiten $4$-angle} ending at $X$ (AR $4$-angle, for short).

Then we assume the following condition.
\begin{assumption}\label{assumption}
Let $M$ be a basic object of a triangulated category $\mathcal{T}$.
\begin{itemize}
\item[(i)]
We have a cluster tilting subcategory $\mathcal{U}$ of $\mathcal{T}$ given by
\begin{align*}
\mathcal{U}:=\add\{\mathbb{S}_{2}^{i}(M)\mid i\in\mathbb{Z}\}.
\end{align*}
\item[(ii)]
The graded algebra $\End_{\mathcal{T}/\mathbb{S}_{2}}(M)$ is generated by homogeneous elements of degree zero and one.
\end{itemize}
The condition $({\rm ii})$ is equivalent to the following condition:
\begin{itemize}
\item[${(\rm ii)^{\prime}}$] There exists a finite quiver $Q$ with a map $\deg : Q_{1}\to \{0,1\}$ such that there exist a surjective morphism $\phi : KQ \to \End_{\mathcal{T}/\mathbb{S}_{2}}(M)$ of graded algebras and the kernel of $\phi$ is contained in the ideal of $KQ$ generated by paths of length at least two.
\end{itemize}
\end{assumption}
The following lemma is a fundamental observation of the quiver $Q$ of $\End_{\mathcal{T}/\mathbb{S}_{2}}(M)$ and right or left minimal almost split morphisms of $M$ in $\mathcal{U}$.
\begin{lemma}\label{arrowalmostsplit}
Under the Assumption \ref{assumption}.
For each $j\in Q_{0}$, let $M^{j}$ be an indecomposable direct summand of $M$ associated with an idempotent $\phi(e_{j})$.
For $j\in Q_{0}$, let
\begin{align*}
f:=(\phi(\alpha)) : \bigoplus_{\alpha\in Q_{1},\,t(\alpha)=j}\mathbb{S}_{2}^{\deg(\alpha)}(M^{s(\alpha)}) \to M^{j}
\end{align*}
be a morphism in $\mathcal{T}$.
Then $f$ is a right minimal almost split morphism of $M^{j}$ in $\mathcal{U}$.
Dually, let 
\begin{align*}
g:=(\phi(\alpha)) : M^{j} \to \bigoplus_{\alpha\in Q_{1},\,s(\alpha)=j}\mathbb{S}_{2}^{-\deg(\alpha)}(M^{t(\alpha)})
\end{align*}
be a morphism in $\mathcal{T}$.
Then $g$ is a left minimal almost split morphism of $M^{j}$ in $\mathcal{U}$.

\if()
For $j\in Q_{0}$, let $X=M^{j}$ and $C_{0}, C_{1}$ be objects in the triangles {\rm (\ref{arend})} of Theorem \ref{iythm3.10}.
Then we have 
\begin{align*}
C_{0}\simeq \bigoplus_{\alpha\in Q_{1},\,t(\alpha)=j}\mathbb{S}_{2}^{\deg(\alpha)}(M^{s(\alpha)}), \quad
C_{1}\simeq \bigoplus_{\alpha\in Q_{1},\,s(\alpha)=j}\mathbb{S}_{2}^{1-\deg(\alpha)}(M^{t(\alpha)}).
\end{align*}
Dually, for objects $C^{0}, C^{1}$ in the triangles {\rm (\ref{arst})} of Theorem \ref{iythm3.10}, we have
\begin{align*}
C^{0}\simeq \bigoplus_{\alpha\in Q_{1},\,s(\alpha)=j}\mathbb{S}_{2}^{-\deg(\alpha)}(M^{t(\alpha)}), \quad
C^{1}\simeq \bigoplus_{\alpha\in Q_{1},\,t(\alpha)=j}\mathbb{S}_{2}^{\deg(\alpha)-1}(M^{s(\alpha)}).
\end{align*}
\fi

\end{lemma}
\begin{proof}
We show that $f$ is a right minimal almost split morphism of $M^{j}$ in $\mathcal{U}$.
Dually, it is shown that $g$ is a left minimal almost split morphism of $M^{j}$ in $\mathcal{U}$.

By definition, $f$ is right minimal and not a retraction.
We denote by $X$ the domain of $f$ and $E:=\End_{\mathcal{T}/\mathbb{S}_{2}}(M)$.
Since $Q$ is the quiver of $E$,
$f$ induces a surjective morphism $\Hom_{\mathcal{T}/\mathbb{S}_{2}}(M,X)\to \rad Ee_{j}$.
Since $\rad Ee_{j}=\rad\left( \Hom_{\mathcal{T}/\mathbb{S}_{2}}(M,M^{j}) \right) =\bigoplus_{i\in\mathbb{Z}}J_{\mathcal{T}}(\mathbb{S}_{2}^{i}(M),M^{j})$ holds,
we have a surjective map $f^{\ast} : \Hom_{\mathcal{T}}(Z,X)\to J_{\mathcal{T}}(Z,M^{j})$ for any $Z\in \mathcal{U}$.
\end{proof}
Before stating the main theorem of this subsection, we need the following definition.
Let $Q$ be a finite quiver with a map $\deg : Q_{1} \to \{0,1\}$.
We define a quiver $Q^{\ast}$ by $Q^{\ast}_{0}=Q_{0}$ and $Q^{\ast}_{1}=\{\, \alpha\in Q_{1} \mid \deg(\alpha)=0 \,\}  \sqcup \{\, \alpha^{\ast}: t(\alpha) \to s(\alpha) \mid \alpha \in Q_{1}, \,\deg(\alpha)=1 \,\}$.
\begin{defprop}\label{degacyclic}
Let $Q$ be a finite quiver with a map $\deg : Q_{1} \to \{0,1\}$.
We call a quiver $Q$  {\it $\deg$-acyclic} if one of the following equivalent conditions holds.
\begin{itemize}
\item[(a)]
The quiver $Q^{\ast}$ is acyclic.
\item[(b)]
There exists an order $\{1,2,\ldots,l\}$ on $Q_{0}$ which satisfies the following conditions:
for any arrow $\alpha: i \to j$ in $Q$, if $\deg(\alpha)=0$, then $j<i$, and if $\deg(\alpha)=1$, then $i<j$.
\end{itemize}
\end{defprop}
The following is the main theorem of this subsection.
\begin{theorem}\label{generatecond}
Under the Assumption \ref{assumption}.
If the quiver $Q$ is $\deg$-acyclic, then we have $\thick_{\mathcal{T}}M=\mathcal{T}$.
\end{theorem}
\begin{proof}
Let $\{1,2,\ldots,l\}$ be an order on $Q_{0}$ which satisfies the condition of Definition-Proposition \ref{degacyclic} (b).
Let $M=\bigoplus_{j=1}^{l}M^{j}$ be an indecomposable direct decomposition of $M$ such that each $M^{j}$ corresponds with a vertex $j \in \{1,2,\ldots,l\}=Q_{0}$.
We show that $\mathbb{S}_{2}^{i}(M^{j})\in\thick_{\mathcal{T}}M$ by an induction on $i$ and $j$.

Let $i\geq 1$.
Assume that $\mathbb{S}_{2}^{k}(M)\in\thick_{\mathcal{T}}M$ for $0\leq k \leq i-1$ and $\mathbb{S}_{2}^{i}(M^{k})\in\thick_{\mathcal{T}}M$ for $0\leq k \leq j-1$, where $M^{0}:=0$.
We show that $\mathbb{S}_{2}^{i}(M^{j})\in\thick_{\mathcal{T}}M$.
By Theorem \ref{iythm3.10}, we have an AR $4$-angle ending at $M^{j}$
\begin{align}\label{c0c1}
\mathbb{S}_{2}(M^{j}) \xto{g} C_{1} \to X_{1} \to \mathbb{S}_{2}(M^{j})[1], \quad
X_{1} \to C_{0} \xto{f} M^{j} \to X_{1}[1],
\end{align}
where $f$ is a right minimal almost split of $M^{j}$ in $\mathcal{U}$ and $g$ is a left minimal almost split of $\mathbb{S}_{2}(M^{j})$ in $\mathcal{U}$.
By Lemma \ref{arrowalmostsplit} and a uniqueness of a right (resp, left) minimal almost split morphism, we have 
\begin{align*}
C_{0}\simeq \bigoplus_{\alpha\in Q_{1},\,t(\alpha)=j}\mathbb{S}_{2}^{\deg(\alpha)}(M^{s(\alpha)}), \quad
C_{1} \simeq \bigoplus_{\alpha\in Q_{1},\,s(\alpha)=j}\mathbb{S}_{2}^{1-\deg(\alpha)}(M^{t(\alpha)}).
\end{align*}
By applying $\mathbb{S}_{2}^{i-1}$ to (\ref{c0c1}), we have an AR $4$-angle ending at $\mathbb{S}_{2}^{i-1}(M^{j})$.
Since $\{1,2,\ldots,l\}=Q_{0}$ satisfies the condition of Definition-Proposition \ref{degacyclic} (b) and by the inductive hypothesis, we have $\mathbb{S}_{2}^{i-1}(C_{0}), \mathbb{S}_{2}^{i-1}(C_{1})\in\thick_{\mathcal{T}}M$.
Thus we have $\mathbb{S}_{2}^{i}(M^{j})\in\thick_{\mathcal{T}}M$ and $\add\{ \mathbb{S}_{2}^{i}(M) \mid i\geq 0 \}\subset \thick_{\mathcal{T}}M$ holds.
An inclusion $\add\{ \mathbb{S}_{2}^{i}(M) \mid i\leq 0 \}\subset \thick_{\mathcal{T}}M$ follows from the dual property of Theorem \ref{iythm3.10} and a similar argument.
Therefore we have $\mathcal{U}\subset\thick_{\mathcal{T}}M$.
By Proposition \ref{iythm3.1}, we have the assertion.
\end{proof}
We end this subsection with the following proposition which calculates the global dimension of the endomorphism algebra $\End_{\mathcal{T}}(M)$.
\begin{proposition}\label{gldimendm}
Under the Assumption \ref{assumption}, suppose that $\Hom_{\mathcal{T}}(M,M[-1])=0$.
Then the global dimension of $\End_{\mathcal{T}}(M)$ is at most two.
\end{proposition}
\begin{proof}
Let $X$ be an indecomposable direct summand of $M$.
Take an AR $4$-angle ending at $X$
\begin{align*}
\mathbb{S}_{2}(X) \to C_{1} \to Y \to \mathbb{S}_{2}(X)[1], \quad Y \to C_{0} \to X \to Y[1].
\end{align*}
By applying the functor $\Hom_{\mathcal{T}}(M,-)$ to the first triangle, we have 
\[ \Hom_{\mathcal{T}}(M,C_{1})\simeq \Hom_{\mathcal{T}}(M,Y), \]
since $\mathcal{U}$ is a cluster tilting subcategory and $\End_{\mathcal{T}}(M)$ is positively graded.
By applying the functor $\Hom_{\mathcal{T}}(M,-)$ to the second triangle, since $\Hom_{\mathcal{T}}(M,M[-1])=0$, we have an exact sequence of $\End_{\mathcal{T}}(M)$-modules
\[ 0\to\Hom_{\mathcal{T}}(M,C_{1})\to\Hom_{\mathcal{T}}(M,C_{0})\to\Hom_{\mathcal{T}}(M,M^{j}). \]
By Lemma \ref{arrowalmostsplit}, we have $C_{0}, C_{1}\in\add\{\, \mathbb{S}_{2}^{i}(M)\mid i=0, 1 \,\}$.
Since $\End_{\mathcal{T}}(M)$ is positively graded, the $\End_{\mathcal{T}}(M)$-modules $\Hom_{\mathcal{T}}(M,C_{0})$ and $\Hom_{\mathcal{T}}(M,C_{1})$ are projective $\End_{\mathcal{T}}(M)$-modules.
Therefore the projective dimension of the simple $\End_{\mathcal{T}}(M)$-module associated with $X$ is at most two, and we have the assertion.
\end{proof}
\subsection{A cluster tilting subcategory of $\USub^{\mathbb{Z}}\Pi(w)$}\label{CTsubcatPiw}
Let $\mbf{w}=s_{u_{1}}s_{u_{2}}\cdots s_{u_{l}}$ be a reduced expression of $w\in W_{Q}$, and put
\begin{align*}
M(\mbf{w})^{i}=M^{i}=(\Pi/I(s_{u_{1}}s_{u_{2}}\cdots s_{u_{i}}))e_{u_{i}}, \quad M(\mbf{w})=M=\bigoplus_{i=1}^{l}M(\mbf{w})^{i}.
\end{align*}
Whenever there is no danger of confusion, we denote $M(\mbf{w})^{i}$ and $M(\mbf{w})$ by $M^{i}$ and $M$, respectively.
In this subsection, we show that the object $M$ of $\USub^{\mathbb{Z}}\Pi(w)$ is a silting object.
Note that by Proposition \ref{serresub}, $\underline{\Sub}^{\mathbb{Z}}\Pi(w)$ has a Serre functor $\mathbb{S}=[2]\circ(-1)$, and hence we have $\mathbb{S}_{2}=(-1)$.
Let \[\mathcal{U}:=\add\{\,M(i)\mid i\in\mathbb{Z} \,\}\]
be the full subcategory of $\USub^{\mathbb{Z}}\Pi(w)$.
\begin{lemma}\label{2clustertilting}
$\mathcal{U}$ is a cluster tilting subcategory of $\USub^{\mathbb{Z}}\Pi(w)$.
\end{lemma}
\begin{proof}
Let $X\in\Sub^{\mathbb{Z}}\Pi(w)$. 
Since $M$ and $X$ are finite dimensional, there exists an integer $N>0$ such that $\Hom_{\Pi(w)}^{\mathbb{Z}}(M,X(i))=\Hom_{\Pi(w)}^{\mathbb{Z}}(X,M(i))=0$ for any $i>|N|$.
This means that $\mathcal{U}$ is functorially finite in $\USub^{\mathbb{Z}}\Pi(w)$.
Since $\mathbb{S}=[2](-1)$ is a Serre functor on $\underline{\Sub}^{\mathbb{Z}}\Pi(w)$, we have \[\mathcal{U}[-1]{}^{\bot}={}^{\bot}\mathcal{U}[1].\]
By Proposition \ref{birs} (d), $\underline{\Hom}_{\Pi(w)}(M,M[1])=0$ holds.
Therefore we have an equality $\underline{\Hom}_{\Pi(w)}^{\mathbb{Z}}(M, M[1](i))=0$ for any integer $i$.
This means $\mathcal{U}\subset{}^{\bot}\mathcal{U}[1].$
Let $X\in\Sub^{\mathbb{Z}}\Pi(w)$ be an indecomposable object such that $X\in{}^{\bot}\mathcal{U}[1]$ in $\underline{\Sub}^{\mathbb{Z}}\Pi(w)$.
By forgetting gradings, we have $\underline{\Hom}_{\Pi(w)}(M,X[1])=0$.
Since $M$ is a cluster tilting object in $\underline{\Sub}\,\Pi(w)$, $X$ is isomorphic to some indecomposable direct summand of $M$ in $\mod\,\Pi(w)$.
By Lemma \ref{ungradediso}, we have $X\in\mathcal{U}$.
\end{proof}
Next we describe the quiver of $\End_{\Pi(w)}(M(\mbf{w}))$.
\begin{definition}\cite{BIRSc}\label{qw}
Let $w$ be an element of $W$.
We define a quiver $Q(\mbf{w})$ associated with a reduced expression ${\mbf{w}}=s_{u_1}s_{u_2}\cdots s_{u_l}$ of $w$ as follows:
\begin{itemize}
	\item vertices: $Q(\mbf{w})_{0}=\{ 1, 2, \ldots, l \}$.
	
	A vertex $1\leq i \leq l$ in $Q(\mbf{w})$ is said to be {\it type $u \in Q_{0}$} if $u_{i}=u$.
	\item arrows: 
		\begin{itemize}
		
		\item[(a1)]
		For each $u\in \supp(w)$, draw an arrow from $j$ to $i$, where $i,j$ are vertices of type $u$, $i<j$, and there is no vertex of type $u$ between $i$ and $j$ {\rm(}we call these arrows {\it going to the left} {\rm)}.
				
		\item[(a2)]
		For each arrow $\alpha : u \to v \in Q_{1}$, draw an arrow $\alpha_{i}$ from $i$ to $j$, where $i<j$, $i$ is a vertex of type $u$, $j$ is a vertex of type $v$, there is no vertex of type $u$ between $i$ and $j$, and $j$ is the biggest vertex of type $v$ before the next vertex of type $u$ (we call these arrows {\it $Q$-arrows}).
		\item[(a3)]
		For each arrow $\alpha : u \to v \in Q_{1}$, draw an arrow $\alpha_{i}^{\ast}$ from $i$ to $j$, where $i<j$, $i$ is a vertex of type $v$, $j$ is a vertex of type $u$, there is no vertex of type $v$ between $i$ and $j$, and $j$ is the biggest vertex of type $u$ before the next vertex of type $v$ (we call these arrows {\it $Q^{\ast}$-arrows}).
		\end{itemize}
	\end{itemize}
We denote by $\underline{Q}(\mbf{w})$ the full subquiver of $Q(\mbf{w})$ whose the set of vertices is $Q(\mbf{w})_{0} \setminus \{p_{u} \mid u\in \supp(w) \}$, where $p_{u}=\Max\{1 \leq j \leq l \mid u_{j}=u\},$ for $u \in \supp(w)$.
\end{definition}
Note that the quiver $Q(\mbf{w})$ depends on the choice of a reduced expression of $w$.
We introduce a map $\deg : Q(\mbf{w})_{1} \to \{0,1\}$.
\begin{definition}\label{gradingofquiver}
We define a map $\deg : Q(\mbf{w})_{1} \to \{0,1\}$ as follows:
\begin{itemize}
\item $\deg(\beta)=1$ if $\beta$ is a $Q^{\ast}$-arrow.
\item $\deg(\beta)=0$ if $\beta$ is a $Q$-arrow or an arrow going to the left.
\end{itemize}
We define a map $\deg$ on $\underline{Q}(\mbf{w})$ as the restriction of $\deg:Q(\mbf{w})_{1}\to \{0,1\}$ to $\underline{Q}(\mbf{w})_{1}$.
\end{definition}
We give an example of a quiver $Q(\mbf{w})$.
\begin{example}\label{exofquvier}
Let $Q$ be the quiver 
	\begin{xy} (0,0)+<0cm,0.2cm>="O",
	"O"+<0cm,0.35cm>="11"*{1},
	"11"+<-0.6cm,-0.7cm>="21"*{2},
	"21"+/r1.2cm/="22"*{3},
	
	\ar_{\alpha}"11"+/dl/;"21"+/u/
	\ar_{\beta}"21"+/r/;"22"+/l/
	\ar^{\gamma}"11"+/dr/;"22"+/u/
	\end{xy}.
Let $w$ be an element of $W_{Q}$ with its expression ${\mbf{w}}=s_{1}s_{2}s_{3}s_{1}s_{3}s_{2}s_{1}$.
The we have the quiver $Q(\mbf{w})$ with a map $\deg : Q(\mbf{w})_{1} \to \{0,1\}$ as follows:
$$
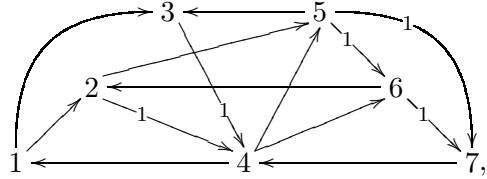
\begin{xy}
	(0,0)+<0cm,0cm>="O",
	"O"+<0cm,2cm>="3"*{3},
	"O"+<-1cm,1cm>="2"*{2},
	"O"+<-2cm,0cm>="1"*{1},
	"1"+<3.0cm,0cm>="4"*{4},
	"3"+<2.0cm,0cm>="5"*{5},
	"2"+<4.0cm,0cm>="6"*{6},
	"4"+<3.0cm,0cm>="7"*{7},
	
	\ar"1"+/ur/;"2"+/dl/
	\ar@(u,l)"1"+/u/;"3"+/l/
	
	\ar"2"+/ur/;"5"+/dl/
	\ar"2"+/dr/;"4"+/ul/|(0.3){1}
	
	\ar"3"+/dr/;"4"+/u/|(0.7){1}
	
	\ar"4"+/l/;"1"+/r/
	\ar"4"+/ur/;"5"+/d/
	\ar"4"+/ur/;"6"+/dl/
	
	\ar"5"+/l/;"3"+/r/
	\ar"5"+/dr/;"6"+/ul/|(0.3){1}
	\ar@(r,u)"5"+/r/;"7"+/u/|(0.3){1}
	
	\ar"6"+/l/;"2"+/r/
	\ar"6"+/dr/;"7"+/ul/|(0.3){1}
	
	\ar"7"+/l/;"4"+/r/
	\end{xy},
$$
where non numbered arrows have degree zero.
\end{example}
We define a morphism of  algebras $\phi : KQ(\mbf{w})\to \End_{\Pi(w)}(M)$ by
\begin{itemize}
	\item[(a0)] For a vertex $i$ of $Q(\mbf{w})$, $\phi(e_{i})$ is an idempotent of $\End_{\Pi(w)}(M)$ associated with $M^{i}$.
	\item[(a1)] For an arrow $\beta : j \to i$ going to the left, $\phi(\beta)$ is the canonical surjection $M^{j} \to M^{i}$.
	\item[(a2)] For a $Q$-arrow $\alpha_{i}: i \to j$ of the arrow $\alpha\in Q_{1}$, $\phi(\alpha_{i})$ is a morphism of $\Pi(w)$-modules from $M^{i}$ to $M^{j}$ given by multiplying $\alpha$ from the right.
	\item[(a3)] For a $Q^{\ast}$-arrow $\alpha_{i}^{\ast}: i \to j$ of the arrow $\alpha\in Q_{1}$, $\phi(\alpha_{i}^{\ast})$ is a morphism of $\Pi(w)$-modules from $M^{i}$ to $M^{j}$ given by multiplying $\alpha^{\ast}$ from the right.
	\end{itemize}
We regard the path algebra $KQ(\mbf{w})$ as a graded algebra by the map $\deg$ of Definition \ref{gradingofquiver}.
The following proposition gives the quiver of the endomorphism algebra
\[\End_{(\USub^{\mathbb{Z}}\Pi(w))/ (-1)}(M)=\bigoplus_{n\in\mathbb{Z}}\underline{\Hom}_{\Pi(w)}^{\mathbb{Z}}(M,M(n))=\underline{\End}_{\Pi(w)}(M).\]

\begin{lemma}\label{gradedhom}
The morphism $\phi : KQ(\mbf{w})\to \End_{\Pi(w)}(M)$ induces a surjective morphism $\underline{\phi} : K\underline{Q}(\mbf{w}) \to \underline{\End}_{\Pi(w)}(M)$ of graded algebras such that the kernel of $\underline{\phi}$ is contained in the ideal of $K\underline{Q}(\mbf{w})$ generated by paths of length at least two.
\end{lemma}
\begin{proof}
The morphism $\phi$ is a morphism of graded algebra, since $\phi$ preserves gradings by the definitions of $\phi$ and the map $\deg$.
The morphism $\phi$ induced a surjective morphism $\underline{\phi}$ of graded algebras by \cite[Theorem I\hspace{-.1em}I\hspace{-.1em}I. 4.1]{BIRSc}.
The kernel of $\underline{\phi}$ is contained in the ideal of $K\underline{Q}(\mbf{w})$ generated by paths of length at least two by \cite[Theorem 6.6]{BIRSm}.
\end{proof}
Then we have the following proposition.
\begin{proposition}\label{assumptionok}
Let ${\mbf{w}}$ be a reduced expression of $w\in W_{Q}$.
Then the following holds.
\begin{itemize}
\item[(a)]
The object $M$ of $\USub^{\mathbb{Z}}\Pi(w)$ satisfies Assumption \ref{assumption}, where the quiver of $\underline{\End}_{\Pi(w)}(M)$ is $\underline{Q}(\mbf{w})$ and a map $\deg$ is given by Definition \ref{gradingofquiver}.
\item[(b)]
The quiver $Q(\mbf{w})$ is $\deg$-acyclic.
In particular, $\underline{Q}(\mbf{w})$ is $\deg$-acyclic.
\end{itemize}
\end{proposition}
\begin{proof}
(\rm a)
This comes from Lemma \ref{2clustertilting} and Lemma \ref{gradedhom}.

(\rm b)
By definition, $(Q(\mbf{w})^{\ast})_{1}$ is a disjoint union of arrows going to the left, $Q$-arrows and reversed arrows of $Q^{\ast}$-arrows.
We define a map \[\psi: (Q(\mbf{w})^{\ast})_{1}\to Q_{0}\sqcup Q_{1}\] by $\psi(\beta)=u$ if $\beta$ is an arrow going to the left associated with a vertex $u\in Q_{0}$ and  $\psi(\beta)=\alpha$ if $\beta$ is a $Q$-arrow or a reversed arrow of $Q^{\ast}$-arrow associated with an arrow $\alpha\in Q_{1}$.
Then $\psi$ extends to a map from the set of all paths in $Q(\mbf{w})^{\ast}$ to the set of all paths in $Q$.
We also denote it by $\psi$.

If there exists a cycle $p$ in $Q(\mbf{w})^{\ast}$, then $\psi(p)$ is a cycle in $Q$.
This is a contradiction. 
\end{proof}
\begin{example}
(a) Let $Q$ be the quiver 
	\begin{xy} (0,0)+<0cm,0.2cm>="O",
	"O"+<0cm,0.35cm>="11"*{1},
	"11"+<-0.6cm,-0.7cm>="21"*{2},
	"21"+/r1.2cm/="22"*{3},
	
	\ar_{\alpha}"11"+/dl/;"21"+/u/
	\ar_{\beta}"21"+/r/;"22"+/l/
	\ar^{\gamma}"11"+/dr/;"22"+/u/
	\end{xy}.
Put $c=s_{1}s_{2}s_{3}$.
Let $w=c^{4}s_{1}=s_{1}s_{2}s_{3}s_{1}s_{2}s_{3}s_{1}s_{2}s_{3}s_{1}s_{2}s_{3}s_{1}$.
Then we have the quiver $Q(\mbf{w})$ and as follows:
$$
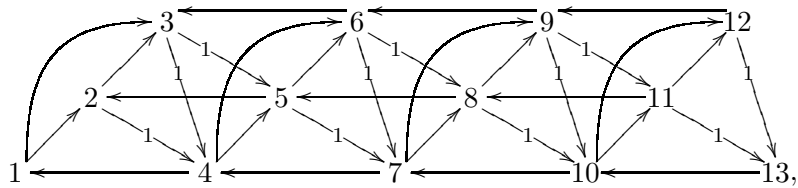
\begin{xy}
	(0,0)+<0cm,0cm>="O",
	"O"+<0cm,2cm>="3"*{3},
	"O"+<-1cm,1cm>="2"*{2},
	"O"+<-2cm,0cm>="1"*{1},
	"1"+<2.5cm,0cm>="4"*{4},
	"2"+<2.5cm,0cm>="5"*{5},
	"3"+<2.5cm,0cm>="6"*{6},
	"4"+<2.5cm,0cm>="7"*{7},
	"5"+<2.5cm,0cm>="8"*{8},
	"6"+<2.5cm,0cm>="9"*{9},	
	"7"+<2.5cm,0cm>="10"*{10},
	"8"+<2.5cm,0cm>="11"*{11},
	"9"+<2.5cm,0cm>="12"*{12},	
	"10"+<2.5cm,0cm>="13"*{13},
	
	\ar"1"+/ur/;"2"+/dl/
	\ar"2"+/ur/;"3"+/dl/
	
	\ar"4"+/ur/;"5"+/dl/
	\ar"5"+/ur/;"6"+/dl/
	
	\ar"7"+/ur/;"8"+/dl/
	\ar"8"+/ur/;"9"+/dl/
	
	\ar"10"+/ur/;"11"+/dl/
	\ar"11"+/ur/;"12"+/dl/
	
	\ar@(u,l)"1"+/ur/;"3"+/l/
	\ar@(u,l)"4"+/ur/;"6"+/l/
	\ar@(u,l)"7"+/ur/;"9"+/l/
	\ar@(u,l)"10"+/ur/;"12"+/l/
	
	\ar"2"+/dr/;"4"+/ul/|{1}
	\ar"3"+/dr/;"5"+/ul/|(0.3){1}
	\ar"3"+/d/;"4"+/u/|(0.3){1}
	
	\ar"5"+/dr/;"7"+/ul/|{1}
	\ar"6"+/dr/;"8"+/ul/|(0.3){1}
	\ar"6"+/d/;"7"+/u/|(0.3){1}
	
	\ar"8"+/dr/;"10"+/ul/|{1}
	\ar"9"+/dr/;"11"+/ul/|(0.3){1}
	\ar"9"+/d/;"10"+/u/|(0.3){1}
	
	\ar"11"+/dr/;"13"+/ul/|{1}
	\ar"12"+/d/;"13"+/u/|(0.3){1}
	
	\ar"4"+/l/;"1"+/r/
	\ar"7"+/l/;"4"+/r/
	\ar"10"+/l/;"7"+/r/
	\ar"13"+/l/;"10"+/r/
	
	\ar"5"+/l/;"2"+/r/
	\ar"8"+/l/;"5"+/r/
	\ar"11"+/l/;"8"+/r/
	
	\ar"6"+/ul/;"3"+/ur/
	\ar"9"+/ul/;"6"+/ur/
	\ar"12"+/ul/;"9"+/ur/
\end{xy},
$$
where non numbered arrows have degree zero.
Then we have the quiver $Q(\mbf{w})^{\ast}$ of $Q(\mbf{w})$
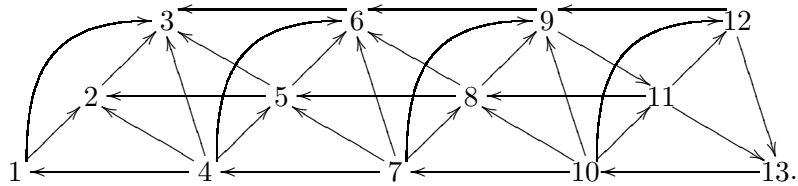
$$
\begin{xy}
	(0,0)+<0cm,0cm>="O",
	"O"+<0cm,2cm>="3"*{3},
	"O"+<-1cm,1cm>="2"*{2},
	"O"+<-2cm,0cm>="1"*{1},
	"1"+<2.5cm,0cm>="4"*{4},
	"2"+<2.5cm,0cm>="5"*{5},
	"3"+<2.5cm,0cm>="6"*{6},
	"4"+<2.5cm,0cm>="7"*{7},
	"5"+<2.5cm,0cm>="8"*{8},
	"6"+<2.5cm,0cm>="9"*{9},	
	"7"+<2.5cm,0cm>="10"*{10},
	"8"+<2.5cm,0cm>="11"*{11},
	"9"+<2.5cm,0cm>="12"*{12},	
	"10"+<2.5cm,0cm>="13"*{13},
	
	\ar"1"+/ur/;"2"+/dl/
	\ar"2"+/ur/;"3"+/dl/
	
	\ar"4"+/ur/;"5"+/dl/
	\ar"5"+/ur/;"6"+/dl/
	
	\ar"7"+/ur/;"8"+/dl/
	\ar"8"+/ur/;"9"+/dl/
	
	\ar"10"+/ur/;"11"+/dl/
	\ar"11"+/ur/;"12"+/dl/
	
	\ar@(u,l)"1"+/ur/;"3"+/l/
	\ar@(u,l)"4"+/ur/;"6"+/l/
	\ar@(u,l)"7"+/ur/;"9"+/l/
	\ar@(u,l)"10"+/ur/;"12"+/l/
	
	\ar"4"+/ul/;"2"+/dr/
	\ar"5"+/ul/;"3"+/dr/
	\ar"4"+/u/;"3"+/d/
	
	\ar"7"+/ul/;"5"+/dr/
	\ar"8"+/ul/;"6"+/dr/
	\ar"7"+/u/;"6"+/d/
	
	\ar"10"+/ul/;"8"+/dr/
	\ar"9"+/dr/;"11"+/ul/
	\ar"10"+/u/;"9"+/d/
	
	\ar"11"+/dr/;"13"+/ul/
	\ar"12"+/d/;"13"+/u/
	
	\ar"4"+/l/;"1"+/r/
	\ar"7"+/l/;"4"+/r/
	\ar"10"+/l/;"7"+/r/
	\ar"13"+/l/;"10"+/r/
	
	\ar"5"+/l/;"2"+/r/
	\ar"8"+/l/;"5"+/r/
	\ar"11"+/l/;"8"+/r/
	
	\ar"6"+/ul/;"3"+/ur/
	\ar"9"+/ul/;"6"+/ur/
	\ar"12"+/ul/;"9"+/ur/
\end{xy}.$$
\end{example}
As a result, we have the following theorem.
\begin{theorem}\label{silting}
Let $w\in W_{Q}$.
For any reduced expression ${\mbf{w}}$ of $w$, the object $M=M(\mbf{w})$ is a silting object of $\USub^{\mathbb{Z}}\Pi(w)$.
\end{theorem}
\begin{proof}
By  Theorem \ref{generatecond} and Proposition \ref{assumptionok}, we have $\thick M=\USub^{\mathbb{Z}}\Pi(w)$.

We show that $M$ satisfies 
$\underline{\Hom}_{\Pi(w)}^{\mathbb{Z}}(M,M[j])=0$ for any $j>0$.
By Proposition \ref{birs} (c), $\underline{\Hom}_{\Pi(w)}(M,M[1])=0$ holds.
Therefore we have $\underline{\Hom}_{\Pi(w)}^{\mathbb{Z}}(M,M[1])=0$.
Assume that $j>1$.
By Proposition \ref{serresub}, we have \[\underline{\Hom}_{\Pi(w)}^{\mathbb{Z}}(M,M[j])\simeq \kD\underline{\Hom}_{\Pi(w)}^{\mathbb{Z}}(M,M[2-j](-1)).\]
Since $2-j\leq 0$ and $\Pi(w)$ is positively graded, we have $\Omega^{-(2-j)}(M)\in\mod^{\geq 0}\Pi(w)$.
Therefore $\Omega^{-(2-j)}(M)(-1)\in\mod^{\geq 1}\Pi(w)$ holds.
Since $M$ is generated by $(M)_{0}$ as a $\Pi(w)$-module,
we have $\Hom_{\Pi(w)}^{\mathbb{Z}}(M,\Omega^{-(2-j)}(M)(-1))=0$.
This means $\underline{\Hom}_{\Pi(w)}^{\mathbb{Z}}(M,M[2-j](-1))=0$ for $j>1$.
\end{proof}
Note that $M(\mbf{w})$ is not a tilting object of $\USub^{\mathbb{Z}}\Pi(w)$ in general.
\begin{example}\label{exnottilting}
Let $Q$ be a quiver
	\begin{xy} (0,0)="O",
	"O"+<0cm,0.35cm>="11"*{1},
	"11"+<-0.8cm,-0.7cm>="21"*{2},
	"21"+/r1.6cm/="22"*{3},
	
	\ar"11"+/dl/;"21"+/u/
	\ar"21"+/r/;"22"+/l/
	\ar"11"+/dr/;"22"+/u/
	\end{xy}.
Then we have a graded algebra $\Pi=\Pi e_1\oplus\Pi e_2\oplus\Pi e_3$, and these are represented by their radical filtrations as follows:
	$$
	\Pi e_{1}=\begin{xy} (0,0)="O",
	"O"+<0cm,1.2cm>="11"*{\bf{1}},
	"11"+<-0.3cm,-0.5cm>="21"*{2},
	"21"+<-0.3cm,-0.5cm>="31"*{3},
	"31"+<-0.3cm,-0.5cm>="41"*{1},
	"41"+<-0.3cm,-0.5cm>="51"*{2},
	"51"+<-0.3cm,-0.5cm>="61"*{3},
	"61"+<-0.3cm,-0.5cm>="71",
	"21"+/r0.6cm/="22"*{3},
	"31"+/r0.6cm/="32"*{1},
	"32"+/r0.6cm/="33"*{2},
	"41"+/r0.6cm/="42"*{2},
	"42"+/r0.6cm/="43"*{3},
	"43"+/r0.6cm/="44"*{1},
	"51"+/r0.6cm/="52"*{3},
	"52"+/r0.6cm/="53"*{1},
	"53"+/r0.6cm/="54"*{2},
	"54"+/r0.6cm/="55"*{3},
	"61"+/r0.6cm/="62"*{1},
	"62"+/r0.6cm/="63"*{2},
	"63"+/r0.6cm/="64"*{3},
	"64"+/r0.6cm/="65"*{1},
	"65"+/r0.6cm/="66"*{2},
	"71"+/r0.6cm/="72",
	"72"+/r0.6cm/="73",
	"73"+/r0.6cm/="74",
	"74"+/r0.6cm/="75",
	"75"+/r0.6cm/="76",
	"76"+/r0.6cm/="77",
	
	\ar@{-}"21"+/dr/;"32"+/u/<-2pt>
	\ar@{-}"22"+/dl/;"32"+/u/<2pt>
	\ar@{-}"22"+/dr/;"33"+/u/<-2pt>
	\ar@{-}"31"+/dl/;"41"+/u/<2pt>
	\ar@{-}"31"+/dr/;"42"+/u/<-2pt>
	\ar@{-}"33"+/dr/;"44"+/u/<-2pt>
	\ar@{-}"42"+/dr/;"53"+/u/<-2pt>
	\ar@{-}"43"+/dl/;"53"+/u/<2pt>
	\ar@{-}"43"+/dr/;"54"+/u/<-2pt>
	\ar@{-}"51"+/dr/;"62"+/u/<-2pt>
	\ar@{-}"52"+/dl/;"62"+/u/<2pt>
	\ar@{-}"52"+/dr/;"63"+/u/<-2pt>
	\ar@{-}"54"+/dr/;"65"+/u/<-2pt>
	\ar@{-}"55"+/dl/;"65"+/u/<2pt>
	\ar@{-}"55"+/dr/;"66"+/u/<-2pt>
	\ar@{-}"61"+/dl/;"71"+/u/<2pt>
	\ar@{-}"61"+/dr/;"72"+/u/<-2pt>
	\ar@{-}"63"+/dr/;"74"+/u/<-2pt>
	\ar@{-}"64"+/dl/;"74"+/u/<2pt>
	\ar@{-}"64"+/dr/;"75"+/u/<-2pt>
	\ar@{-}"66"+/dr/;"77"+/u/<-2pt>
	\end{xy},
\quad
	\Pi e_{2}=\begin{xy} (0,0)="O",
	"O"+<0cm,1.2cm>="11"*{{\bf 2}},
	"11"+<-0.3cm,-0.5cm>="21"*{3},
	"21"+<-0.3cm,-0.5cm>="31"*{1},
	"31"+<-0.3cm,-0.5cm>="41"*{2},
	"41"+<-0.3cm,-0.5cm>="51"*{3},
	"51"+<-0.3cm,-0.5cm>="61"*{1},
	"61"+<-0.3cm,-0.5cm>="71",
	"21"+/r0.6cm/="22"*{{\bf 1}},
	"31"+/r0.6cm/="32"*{2},
	"32"+/r0.6cm/="33"*{3},
	"41"+/r0.6cm/="42"*{3},
	"42"+/r0.6cm/="43"*{1},
	"43"+/r0.6cm/="44"*{2},
	"51"+/r0.6cm/="52"*{1},
	"52"+/r0.6cm/="53"*{2},
	"53"+/r0.6cm/="54"*{3},
	"54"+/r0.6cm/="55"*{1},
	"61"+/r0.6cm/="62"*{2},
	"62"+/r0.6cm/="63"*{3},
	"63"+/r0.6cm/="64"*{1},
	"64"+/r0.6cm/="65"*{2},
	"65"+/r0.6cm/="66"*{3},
	"71"+/r0.6cm/="72",
	"72"+/r0.6cm/="73",
	"73"+/r0.6cm/="74",
	"74"+/r0.6cm/="75",
	"75"+/r0.6cm/="76",
	"76"+/r0.6cm/="77",
	
	\ar@{-}"11"+/dr/;"22"+/u/<-2pt>
	\ar@{-}"21"+/dl/;"31"+/u/<2pt>
	\ar@{-}"21"+/dr/;"32"+/u/<-2pt>
	\ar@{-}"32"+/dr/;"43"+/u/<-2pt>
	\ar@{-}"33"+/dl/;"43"+/u/<2pt>
	\ar@{-}"33"+/dr/;"44"+/u/<-2pt>
	\ar@{-}"41"+/dr/;"52"+/u/<-2pt>
	\ar@{-}"42"+/dl/;"52"+/u/<2pt>
	\ar@{-}"42"+/dr/;"53"+/u/<-2pt>
	\ar@{-}"44"+/dr/;"55"+/u/<-2pt>
	\ar@{-}"51"+/dl/;"61"+/u/<2pt>
	\ar@{-}"51"+/dr/;"62"+/u/<-2pt>
	\ar@{-}"53"+/dr/;"64"+/u/<-2pt>
	\ar@{-}"54"+/dl/;"64"+/u/<2pt>
	\ar@{-}"54"+/dr/;"65"+/u/<-2pt>
	\ar@{-}"62"+/dr/;"73"+/u/<-2pt>
	\ar@{-}"63"+/dl/;"73"+/u/<2pt>
	\ar@{-}"63"+/dr/;"74"+/u/<-2pt>
	\ar@{-}"65"+/dr/;"76"+/u/<-2pt>
	\ar@{-}"66"+/dl/;"76"+/u/<2pt>
	\ar@{-}"66"+/dr/;"77"+/u/<-2pt>
	\end{xy},
\quad
	\Pi e_{3}=\begin{xy} (0,0)="O",
	"O"+<0cm,1.2cm>="11"*{{\bf 3}},
	"11"+<-0.3cm,-0.5cm>="21"*{{\bf 1}},
	"21"+<-0.3cm,-0.5cm>="31"*{2},
	"31"+<-0.3cm,-0.5cm>="41"*{3},
	"41"+<-0.3cm,-0.5cm>="51"*{1},
	"51"+<-0.3cm,-0.5cm>="61"*{2},
	"61"+<-0.3cm,-0.5cm>="71",
	"21"+/r0.6cm/="22"*{{\bf 2}},
	"31"+/r0.6cm/="32"*{3},
	"32"+/r0.6cm/="33"*{{\bf 1}},
	"41"+/r0.6cm/="42"*{1},
	"42"+/r0.6cm/="43"*{2},
	"43"+/r0.6cm/="44"*{3},
	"51"+/r0.6cm/="52"*{2},
	"52"+/r0.6cm/="53"*{3},
	"53"+/r0.6cm/="54"*{1},
	"54"+/r0.6cm/="55"*{2},
	"61"+/r0.6cm/="62"*{3},
	"62"+/r0.6cm/="63"*{1},
	"63"+/r0.6cm/="64"*{2},
	"64"+/r0.6cm/="65"*{3},
	"65"+/r0.6cm/="66"*{1},
	"71"+/r0.6cm/="72",
	"72"+/r0.6cm/="73",
	"73"+/r0.6cm/="74",
	"74"+/r0.6cm/="75",
	"75"+/r0.6cm/="76",
	"76"+/r0.6cm/="77",
	
	\ar@{-}"11"+/dl/;"21"+/u/<2pt>
	\ar@{-}"11"+/dr/;"22"+/u/<-2pt>
	\ar@{-}"22"+/dr/;"33"+/u/<-2pt>
	\ar@{-}"31"+/dr/;"42"+/u/<-2pt>
	\ar@{-}"32"+/dl/;"42"+/u/<2pt>
	\ar@{-}"32"+/dr/;"43"+/u/<-2pt>
	\ar@{-}"41"+/dl/;"51"+/u/<2pt>
	\ar@{-}"41"+/dr/;"52"+/u/<-2pt>
	\ar@{-}"43"+/dr/;"54"+/u/<-2pt>
	\ar@{-}"44"+/dl/;"54"+/u/<2pt>
	\ar@{-}"44"+/dr/;"55"+/u/<-2pt>
	\ar@{-}"52"+/dr/;"63"+/u/<-2pt>
	\ar@{-}"53"+/dl/;"63"+/u/<2pt>
	\ar@{-}"53"+/dr/;"64"+/u/<-2pt>
	\ar@{-}"55"+/dr/;"66"+/u/<-2pt>
	\ar@{-}"61"+/dr/;"72"+/u/<-2pt>
	\ar@{-}"62"+/dl/;"72"+/u/<2pt>
	\ar@{-}"62"+/dr/;"73"+/u/<-2pt>
	\ar@{-}"64"+/dr/;"75"+/u/<-2pt>
	\ar@{-}"65"+/dl/;"75"+/u/<2pt>
	\ar@{-}"65"+/dr/;"76"+/u/<-2pt>
	\end{xy},
	$$
where numbers connected by solid lines are concentrated in the same degree, 
the tops of the $\Pi e_i$ are concentrated in degree $0$,
and the degree zero parts are denoted by bold numbers.

Let $w$ be an element of $W_{Q}$ which has a reduced expression ${\mbf{w}}=s_{3}s_{2}s_{1}s_{2}s_{3}s_{2}$.
Then we have a graded algebra, $\Pi(w)=\Pi(w)e_1\oplus\Pi(w)e_2\oplus\Pi(w)e_3$, where
	$$
	\Pi(w) e_{1}=\begin{xy} (0,0)="O",
	"O"+<0cm,0.6cm>="11"*{{\bf 1}},
	"11"+<-0.3cm,-0.5cm>="21"*{2},
	"21"+<-0.3cm,-0.5cm>="31"*{3},
	"21"+/r0.6cm/="22"*{3},
	\end{xy},
\qquad
	\Pi(w)e_{2}=\begin{xy} (0,0)="O",
	"O"+<0cm,1cm>="11"*{{\bf 2}},
	"11"+<-0.3cm,-0.5cm>="21"*{3},
	"21"+<-0.3cm,-0.5cm>="31"*{1},
	"31"+<-0.3cm,-0.5cm>="41"*{2},
	"41"+<-0.3cm,-0.5cm>="51"*{3},
	"21"+/r0.6cm/="22"*{{\bf 1}},
	"31"+/r0.6cm/="32"*{2},
	"32"+/r0.6cm/="33"*{3},
	"41"+/r0.6cm/="42"*{3},
	\ar@{-}"11"+/dr/;"22"+/u/<-2pt>
	\ar@{-}"21"+/dl/;"31"+/u/<2pt>
	\ar@{-}"21"+/dr/;"32"+/u/<-2pt>
	\end{xy},
\qquad
	\Pi(w)e_{3}=\begin{xy} (0,0)="O",
	"O"+<0cm,0.6cm>="11"*{{\bf 3}},
	"11"+<-0.3cm,-0.5cm>="21"*{\bf 1},
	"21"+<-0.3cm,-0.5cm>="31"*{2},
	"31"+<-0.3cm,-0.5cm>="41"*{3},
	"21"+/r0.6cm/="22"*{\bf 2},
	"31"+/r0.6cm/="32"*{3},
	"32"+/r0.6cm/="33"*{\bf 1},
	"41"+/r0.6cm/="42",
	"42"+/r0.6cm/="43",
	"43"+/r0.6cm/="44"*{3},
	\ar@{-}"11"+/dl/;"21"+/u/<2pt>
	\ar@{-}"11"+/dr/;"22"+/u/<-2pt>
	\ar@{-}"22"+/dr/;"33"+/u/<-2pt>
	\end{xy}.
	$$
We have a silting object $M=M(\mbf{w})$ of $\USub^{\mathbb{Z}}\Pi(w)$ as follows:
\[
M=M^{1}\oplus M^{2}\oplus M^{4}={\bf 3} 
\oplus 
	\begin{xy} (0,0)="O",
	"O"+<0.3cm,0.25cm>="11"*{{\bf 2}},
	"11"+<-0.3cm,-0.5cm>="21"*{3},
	\end{xy}
\oplus
	\begin{xy} (0,0)="O",
	"O"+<0cm,0.5cm>="11"*{{\bf 2}},
	"11"+<-0.3cm,-0.5cm>="21"*{3},
	"21"+<-0.3cm,-0.5cm>="31",
	"31"+<-0.3cm,-0.5cm>="41",
	"41"+<-0.3cm,-0.5cm>="51",
	"21"+/r0.6cm/="22"*{{\bf 1}},
	"31"+/r0.6cm/="32",
	"32"+/r0.6cm/="33"*{3},
	
	\ar@{-}"11"+/dr/;"22"+/u/<-2pt>
	\end{xy}.
\]
This $M(\mbf{w})$ is not a tilting object of $\USub^{\mathbb{Z}}\Pi(w)$, since we see that $\underline{\Hom}^{\mathbb{Z}}_{\Pi(w)}(M^{4},\Omega(M^{1}))\neq0$.
Note that another reduced expression of $w$ gives a tilting object of $\USub^{\mathbb{Z}}\Pi(w)$ (see Example \ref{extilt} (a)).
\end{example}
\section{A tilting object in $\USub^{\mathbb{Z}}\Pi(w)$}\label{sectiontilting}
Let $w\in W_{Q}$ and ${\mbf{w}}$ be a reduced expression of $w$.
In this section, we give a sufficient condition on ${\mbf{w}}$ such that $M=M(\mbf{w})$ is a tilting object of $\USub^{\mathbb{Z}}\Pi(w)$.
Throughout this section, by Lemma \ref{airtlem2.1}, without loss of generality assume that $\supp(w)=Q_{0}$.
We first show the following lemma.
\begin{lemma}\label{-2}
If one of the following holds, then we have $\underline{\Hom}_{\Pi(w)}^{\mathbb{Z}}(M,M[j])=0$ for any $j<-1$.
\begin{itemize}
\item[(i)]
There exists a reduced expression of $w$ containing an expression of the Coxeter element of $W_{Q}$ as a subword.
\item[(ii)]
The global dimension of $\Pi(w)_{0}$ is at most one.
\end{itemize}
\end{lemma}
\begin{proof}
By Lemma \ref{euev} (b), we assume that (ii) holds. 
By the definition of $M$, $M$ is in $\mod^{\geq 0}\Pi(w)$.
Therefore by Lemma \ref{A_0}, we have $\Omega^{j}(M)_{0}=0$ for any $j>1$.
Since $M$ is generated by $(M)_{0}$ as a $\Pi(w)$-module, we have $\Hom_{\Pi(w)}^{\mathbb{Z}}(M,\Omega^{j}(M))=0$ for any $j>1$.
This means $\underline{\Hom}_{\Pi(w)}^{\mathbb{Z}}(M,M[j])=0$ for any $j<-1$.
\end{proof}
Next we observe when $\underline{\Hom}^{\mathbb{Z}}_{\Pi(w)}(M, M[-1])=0$ holds.
We need the following lemma.
For a reduced expression $s_{u_1}s_{u_2}\cdots s_{u_l}$, let $I_{k,m}=I(s_{u_{k}}\cdots s_{u_{m}})$ if $k \leq m$ and $I_{k,m}=\Pi$ if $m < k$.
\begin{lemma}\cite[Lemma I\hspace{-.1em}I\hspace{-.1em}I. 1.14]{BIRSc}\label{birslem1.14}
Assume that $s_{u_1}s_{u_2}\cdots s_{u_l}$ is a reduced expression.
Then we have $I_{k+1,m}/I_{1,m} \simeq \Hom_{\Pi}(\Pi/I(s_{u_{1}}\ldots s_{{u_k}}), \Pi/I(s_{u_{1}}\ldots s_{u_m})$.
\end{lemma} 
\begin{proof}
If $Q$ is a non-Dynkin quiver, then the assertion holds by \cite[Lemma I\hspace{-.1em}I\hspace{-.1em}I. 1.14]{BIRSc}.
The assertion also holds when $Q$ is a Dynkin quiver by Lemma \ref{airtlem2.1}.
\end{proof}
We define some notation.
A full subquiver $Q^{\prime}$ of $Q$ is said to be {\it convex} in $Q$ if any path in $Q$ such that its start and target are in $Q^{\prime}$ is a path in $Q^{\prime}$.
For any $u,v\in Q_{0}$, we denote by $Q(u,v)$ the minimal convex  full subquiver of $Q$ containing $u$ and $v$.
Let ${\mbf{w}}=s_{u_{1}}s_{u_{2}}\cdots s_{u_{l}}$ be a reduced expression of $w$.
For any $u\in Q_{0}$, put
\begin{align*}
p_{u}=\max\{1 \leq j \leq l \mid u_{j}=u\}, \quad
m_{u}=\min\{1 \leq j \leq l \mid u_{j}=u\}.
\end{align*}
\begin{definition}\label{PM}
Let ${\mbf{w}}=s_{u_{1}}s_{u_{2}}\cdots s_{u_{l}}$ be a reduced expression of $w\in W_{Q}$ and $S$ be a subset of $Q_{0}$.
\begin{itemize}
\item[(1)]
An expression ${\mbf{w}}$ is {\it $c$-ending on $S$} if for any $u,v\in S$, $p_{u}< p_{v}$ holds whenever there exists an arrow from $u$ to $v$ in $Q$.
\item[(2)]
An expression ${\mbf{w}}$ is {\it $c$-starting on $S$} if for any $u,v\in S$, $m_{u}< m_{v}$ holds whenever there exists an arrow from $u$ to $v$ in $Q$.
\end{itemize}
\end{definition}
The following lemma is an easy observation.
\begin{lemma}\label{startend}
Let $w\in W_{Q}$ and ${\mbf{w}}$ be a reduced expression of $w$.
If ${\mbf{w}}$ is $c$-ending or $c$-starting on $Q_{0}$, then ${\mbf{w}}$ contains an expression of the Coxeter element of $W_{Q}$ as a subword, in particular the global dimension of $\Pi(w)_{0}$ is at most one.
\end{lemma}
The following proposition is important to show the main theorem of this section.
\begin{proposition}\label{-1ij}
Let ${\mbf{w}}=s_{u_{1}}s_{u_{2}}\cdots s_{u_{l}}$ be a reduced expression of $w\in W_{Q}$ and $i,j\in\{1,\dots,l\}\setminus\{p_{u}\mid u\in Q_{0}\}$.
If an expression ${\mbf{w}}$ is $c$-ending on $Q(u_{i},u_{j})_{0}$ or $c$-starting on $Q(u_{i},u_{j})_{0}$, then we have $\Hom_{\Pi(w)}^{\mathbb{Z}}(M^{i},\Omega(M^{j}))=0$.
\end{proposition}
\begin{proof}
By Lemma \ref{birslem1.14} and applying the functor $\Hom_{\Pi(w)}^{\mathbb{Z}}(M^{i},-)$ to an exact sequence $0\to\Omega(M^{j})\to\Pi(w)e_{u_{j}}\to M^{j}\to0$, we have 
\begin{align*}
\Hom_{\Pi(w)}^{\mathbb{Z}}(M^{i},\Omega(M^{j})) \simeq e_{u_{i}}\left(\frac{I_{1,j}\cap I_{i+1,l}}{I(w)}\right)_{0}e_{u_{j}}.
\end{align*}
Therefore it is enough to show that $e_{u_{i}}(I_{1,j}\cap I_{i+1,l})_{0}e_{u_{j}}=0$.

Since $(I_{1,j}\cap I_{i+1,l})_{0}\subset KQ$, if $e_{u_{i}}KQe_{u_{j}}=0$, then we have $e_{u_{i}}(I_{1,j}\cap I_{i+1,l})_{0}e_{u_{j}}=0$.
Assume that $e_{u_{i}}KQe_{u_{j}}\neq0$.
Let $c_{i,j}$ be the Coxeter element of $W_{Q(u_{i},u_{j})}$.
Since $Q(u_{i},u_{j})$ is a full subquiver of $Q$, an expression of $c_{i,j}$ is a subword of an expression of the Coxeter element of $W_{Q}$.
Since $Q(u_{i},u_{j})$ is a minimal convex subquiver of $Q$, $u_{i}$ is a unique source of $Q(u_{i},u_{j})$ and $u_{j}$ is a unique sink of $Q(u_{i},u_{j})$.
Therefore by Lemma \ref{euev} (c), we have $e_{u_{i}}I(c_{i,j})_{0}e_{u_{j}}=0$.

If ${\mbf{w}}$ is $c$-ending on $Q(u_{i},u_{j})_{0}$, then an expression $s_{u_{i+1}}\cdots s_{u_{l}}$ contains an expression of $c_{i,j}$ as a subword, and therefore $e_{u_{i}}(I_{i+1,l})_{0}e_{u_{j}}\subset e_{u_{i}}I(c_{i,j})_{0}e_{u_{j}}=0$ holds.

If ${\mbf{w}}$ is  $c$-starting on $Q(u_{i},u_{j})_{0}$, then an expression $s_{u_{1}}\cdots s_{u_{j}}$ contains an expression of $c_{i,j}$ as a subword, and therefore $e_{u_{i}}(I_{1,j})_{0}e_{u_{j}}\subset e_{u_{i}}I(c_{i,j})_{0}e_{u_{j}}=0$.
We have the assertion.
\end{proof}
Then we show the main theorem of this section.
\begin{theorem}\label{tilting}
Let $w\in W_{Q}$ and ${\mbf{w}}=s_{u_{1}}s_{u_{2}}\cdots s_{u_{l}}$ be a reduced expression of $w$.
Put
\begin{align*}
&M^{i}=(\Pi/I(s_{u_{1}}s_{u_{2}}\cdots s_{u_{i}}))e_{u_{i}}, \quad M=\bigoplus_{i=1}^{l}M^{i}.
\end{align*}
If the expression ${\mbf{w}}$ is $c$-ending on $Q_{0}$ or $c$-starting on $Q_{0}$, then we have the following.
\begin{itemize}
\item[(a)] $M$ is a tilting object of $\USub^{\mathbb{Z}}\Pi(w)$.
\item[(b)] The global dimension of $A=\underline{\End}_{\Pi(w)}^{\mathbb{Z}}(M)$ is at most two.
\item[(c)] We have a triangle equivalence $\USub^{\mathbb{Z}}\Pi(w) \simeq  {\mathsf D}^{{\rm b}}(\mod A)$.
\end{itemize}
\end{theorem}
\begin{proof}
(\rm a)
By Theorem \ref{silting}, Lemmas \ref{-2} and \ref{startend}, we only have to show $\underline{\Hom}_{\Pi(w)}^{\mathbb{Z}}(M,M[-1])=0$.
We show that $\Hom_{\Pi(w)}^{\mathbb{Z}}(M^{i},\Omega(M^{j}))=0$ for any $i,j\in\{1,2,\ldots,l\}\setminus\{p_{u}\mid u\in Q_{0}\}$.
Since ${\mbf{w}}$ is $c$-ending on $Q_{0}$ or $c$-starting on $Q_{0}$, ${\mbf{w}}$ is $c$-ending on $Q(u_{i},u_{j})_{0}$ or $c$-starting on $Q(u_{i},u_{j})_{0}$. 
Therefore, we have $\Hom_{\Pi(w)}^{\mathbb{Z}}(M^{i},\Omega(M^{j}))=0$ by Proposition \ref{-1ij}.

(\rm b)
This comes from (\rm a) and Proposition \ref{gldimendm}.

(\rm c)
This follows from (\rm a), (\rm b) and Theorem \ref{keller}.
\end{proof}
\begin{remark}
The property (\rm b) of Theorem \ref{tilting} was already shown by \cite{ART} in the case when ${\mbf{w}}$ is $c$-ending on $\supp(w)$ (see Theorem \ref{ART}).
\end{remark}
Next we give a more general condition on $\mbf{w}$ such that $M(\mbf{w})$ satisfies $\underline{\Hom}_{\Pi(w)}^{\mathbb{Z}}(M,M[-1])=0$.
For a reduced expression $\mbf{w}$, let $S({\mbf{w}}):=\{u\in Q_{0} \mid p_{u}=m_{u}\}$.
\begin{definition}
A reduced expression ${\mbf{w}}$ satisfies $(\diamondsuit)$ if for any $u,v\in Q_{0}\setminus S({\mbf{w}})$, ${\mbf{w}}$ is $c$-ending on $Q(u,v)_{0}$ or $c$-starting on $Q(u,v)_{0}$.
\end{definition}
Put $J=\{1,2,\ldots,l\}\setminus\{p_{u}\mid u\in Q_{0}\}$.
Note that $\{\,u_{i} \mid i\in J \,\}=Q_{0}\setminus S({\mbf{w}})$ holds.
We have the following theorem.
\begin{theorem}
Let $w\in W_{Q}$.
Assume that the global dimension of $\Pi(w)_{0}$ is at most one.
Let $\mbf{w}=s_{u_{1}}s_{u_{2}}\cdots s_{u_{l}}$ be a reduced expression of $w$ and $M$ be the same object as that in Theorem \ref{tilting}.
If $\mbf{w}$ satisfies $(\diamondsuit)$, then the assertions {\rm (a)}, {\rm (b)} and {\rm (c)} of Theorem \ref{tilting} hold.
\end{theorem}
\begin{proof}
These are shown by the same argument as that in Theorem \ref{tilting} since $\{\,u_{i} \mid i\in J \,\}=Q_{0}\setminus S({\mbf{w}})$ holds.
\end{proof}
An example of a reduced expression which satisfies $(\diamondsuit)$ but is neither $c$-ending nor $c$-starting on $Q_{0}$ is given in Example \ref{extilt} (c).
We end this section by giving some examples.
\begin{example}\label{extilt}
(\rm a) Let $Q$ be a quiver
	\begin{xy} (0,0)="O",
	"O"+<0cm,0.35cm>="11"*{1},
	"11"+<-0.8cm,-0.7cm>="21"*{2},
	"21"+/r1.6cm/="22"*{3},
	
	\ar"11"+/dl/;"21"+/u/
	\ar"21"+/r/;"22"+/l/
	\ar"11"+/dr/;"22"+/u/
	\end{xy}.
Let $w$ be an element of $W_{Q}$ which has a reduced expression ${\mbf{w}}=s_{3}s_{2}s_{1}s_{3}s_{2}s_{3}$.
Note that this $w$ is the same element as that in Example \ref{exnottilting}.
The expression ${\mbf{w}}$ is $c$-ending on $Q_{0}$.
Then we have a graded algebra, $\Pi(w)=\Pi(w)e_1\oplus\Pi(w)e_2\oplus\Pi(w)e_3$, where
	$$
	\Pi(w) e_{1}=\begin{xy} (0,0)="O",
	"O"+<0cm,0.6cm>="11"*{{\bf 1}},
	"11"+<-0.3cm,-0.5cm>="21"*{2},
	"21"+<-0.3cm,-0.5cm>="31"*{3},
	"21"+/r0.6cm/="22"*{3},
	\end{xy},
\qquad
	\Pi(w)e_{2}=\begin{xy} (0,0)="O",
	"O"+<0cm,1cm>="11"*{{\bf 2}},
	"11"+<-0.3cm,-0.5cm>="21"*{3},
	"21"+<-0.3cm,-0.5cm>="31"*{1},
	"31"+<-0.3cm,-0.5cm>="41"*{2},
	"41"+<-0.3cm,-0.5cm>="51"*{3},
	"21"+/r0.6cm/="22"*{{\bf 1}},
	"31"+/r0.6cm/="32"*{2},
	"32"+/r0.6cm/="33"*{3},
	"41"+/r0.6cm/="42"*{3},
	\ar@{-}"11"+/dr/;"22"+/u/<-2pt>
	\ar@{-}"21"+/dl/;"31"+/u/<2pt>
	\ar@{-}"21"+/dr/;"32"+/u/<-2pt>
	\end{xy},
\qquad
	\Pi(w)e_{3}=\begin{xy} (0,0)="O",
	"O"+<0cm,0.6cm>="11"*{{\bf 3}},
	"11"+<-0.3cm,-0.5cm>="21"*{\bf 1},
	"21"+<-0.3cm,-0.5cm>="31"*{2},
	"31"+<-0.3cm,-0.5cm>="41"*{3},
	"21"+/r0.6cm/="22"*{\bf 2},
	"31"+/r0.6cm/="32"*{3},
	"32"+/r0.6cm/="33"*{\bf 1},
	"41"+/r0.6cm/="42",
	"42"+/r0.6cm/="43",
	"43"+/r0.6cm/="44"*{3},
	\ar@{-}"11"+/dl/;"21"+/u/<2pt>
	\ar@{-}"11"+/dr/;"22"+/u/<-2pt>
	\ar@{-}"22"+/dr/;"33"+/u/<-2pt>
	\end{xy}.
	$$
We have a tilting object $M=M(\mbf{w})$ of $\USub^{\mathbb{Z}}\Pi(w)$ as follows:
\[
M={\bf 3} 
\oplus 
	\begin{xy} (0,0)="O",
	"O"+<0.3cm,0.25cm>="11"*{{\bf 2}},
	"11"+<-0.3cm,-0.5cm>="21"*{3},
	\end{xy}
\oplus
	\begin{xy} (0,0)="O",
	"O"+<0cm,0.6cm>="11"*{{\bf 3}},
	"11"+<-0.3cm,-0.5cm>="21"*{\bf 1},
	"21"+<-0.3cm,-0.5cm>="31"*{2},
	"31"+<-0.3cm,-0.5cm>="41"*{3},
	"21"+/r0.6cm/="22"*{\bf 2},
	"31"+/r0.6cm/="32"*{3},
	\ar@{-}"11"+/dl/;"21"+/u/<2pt>
	\ar@{-}"11"+/dr/;"22"+/u/<-2pt>
	\end{xy}.
\]
The endomorphism algebra $\underline{\End}_{\Pi(w)}^{\mathbb{Z}}(M)$ is given by the following quiver with relations
\[
\Delta =\left[\xymatrix{ \bullet \ar[r]^(0.6){a} & \bullet \ar[r]^{b} & \bullet}\right], \quad ab=0.
\]

(\rm b)
Let $Q$ be the same quiver as that in $(\rm a)$ and $w$ be an element of $W_{Q}$ with its expression ${\mbf{w}}=s_{1}s_{2}s_{3}s_{1}s_{3}s_{2}s_{1}$.
This expression ${\mbf{w}}$ is a reduced expression and $c$-starting on $Q_{0}$.
Then we have
	$$
	\Pi(w)e_{1}=\begin{xy} (0,0)="O",
	"O"+<0cm,1cm>="11"*{{\bf 1}},
	"11"+<-0.3cm,-0.5cm>="21"*{2},
	"21"+<-0.3cm,-0.5cm>="31"*{3},
	"31"+<-0.3cm,-0.5cm>="41"*{1},
	"41"+<-0.3cm,-0.5cm>="51"*{2},
	"21"+/r0.6cm/="22"*{3},
	"31"+/r0.6cm/="32"*{1},
	"32"+/r0.6cm/="33"*{2},
	"41"+/r0.6cm/="42"*{2},
	"42"+/r0.6cm/="43",
	"43"+/r0.6cm/="44"*{1},
	\ar@{-}"21"+/dr/;"32"+/u/<-2pt>
	\ar@{-}"22"+/dl/;"32"+/u/<2pt>
	\ar@{-}"22"+/dr/;"33"+/u/<-2pt>
	\ar@{-}"31"+/dl/;"41"+/u/<2pt>
	\ar@{-}"31"+/dr/;"42"+/u/<-2pt>
	\ar@{-}"33"+/dr/;"44"+/u/<-2pt>
	\end{xy},
\qquad
	\Pi(w)e_{2}=\begin{xy} (0,0)="O",
	"O"+<0cm,1cm>="11"*{{\bf 2}},
	"11"+<-0.3cm,-0.5cm>="21"*{3},
	"21"+<-0.3cm,-0.5cm>="31"*{1},
	"31"+<-0.3cm,-0.5cm>="41"*{2},
	"41"+<-0.3cm,-0.5cm>="51",
	"21"+/r0.6cm/="22"*{{\bf 1}},
	"31"+/r0.6cm/="32"*{2},
	"32"+/r0.6cm/="33"*{3},
	"41"+/r0.6cm/="42",
	"42"+/r0.6cm/="43"*{1},
	"43"+/r0.6cm/="44"*{2},
	"51"+/r0.6cm/="52",
	"52"+/r0.6cm/="53",
	"53"+/r0.6cm/="54",
	"54"+/r0.6cm/="55"*{1},
	
	\ar@{-}"11"+/dr/;"22"+/u/<-2pt>
	\ar@{-}"21"+/dl/;"31"+/u/<2pt>
	\ar@{-}"21"+/dr/;"32"+/u/<-2pt>
	\ar@{-}"32"+/dr/;"43"+/u/<-2pt>
	\ar@{-}"33"+/dl/;"43"+/u/<2pt>
	\ar@{-}"33"+/dr/;"44"+/u/<-2pt>
	\ar@{-}"44"+/dr/;"55"+/u/<-2pt>
	\end{xy},
\qquad
	\Pi(w)e_{3}=\begin{xy} (0,0)="O",
	"O"+<0cm,0.6cm>="11"*{{\bf 3}},
	"11"+<-0.3cm,-0.5cm>="21"*{\bf 1},
	"21"+<-0.3cm,-0.5cm>="31"*{2},
	"21"+/r0.6cm/="22"*{\bf 2},
	"31"+/r0.6cm/="32",
	"32"+/r0.6cm/="33"*{\bf 1},
	\ar@{-}"11"+/dl/;"21"+/u/<2pt>
	\ar@{-}"11"+/dr/;"22"+/u/<-2pt>
	\ar@{-}"22"+/dr/;"33"+/u/<-2pt>
	\end{xy}.
	$$
A tilting object $M=M(\mbf{w})$ of $\USub^{\mathbb{Z}}\Pi(w)$ is described as follows:
\[
M={\bf 1}
\oplus 
	\begin{xy} (0,0)="O",
	"O"+<0.3cm,0.25cm>="11"*{{\bf 2}},
	"11"+<0.3cm,-0.5cm>="22"*{\bf 1},
	\ar@{-}"11"+/dr/;"22"+/u/<-2pt>
	\end{xy}
\oplus
	\begin{xy} (0,0)="O",
	"O"+<0cm,0.6cm>="11"*{{\bf 3}},
	"11"+<-0.3cm,-0.5cm>="21"*{\bf 1},
	"21"+<-0.3cm,-0.5cm>="31",
	"21"+/r0.6cm/="22"*{\bf 2},
	"31"+/r0.6cm/="32",
	"32"+/r0.6cm/="33"*{\bf 1},
	\ar@{-}"11"+/dl/;"21"+/u/<2pt>
	\ar@{-}"11"+/dr/;"22"+/u/<-2pt>
	\ar@{-}"22"+/dr/;"33"+/u/<-2pt>
	\end{xy}
\oplus
	\begin{xy} (0,0)="O",
	"O"+<0cm,0.8cm>="11"*{{\bf 1}},
	"11"+<-0.3cm,-0.5cm>="21"*{2},
	"21"+<-0.3cm,-0.5cm>="31",
	"31"+<-0.3cm,-0.5cm>="41",
	"41"+<-0.3cm,-0.5cm>="51",
	"21"+/r0.6cm/="22"*{3},
	"31"+/r0.6cm/="32"*{1},
	"32"+/r0.6cm/="33"*{2},
	"41"+/r0.6cm/="42",
	"42"+/r0.6cm/="43",
	"43"+/r0.6cm/="44"*{1},
	\ar@{-}"21"+/dr/;"32"+/u/<-2pt>
	\ar@{-}"22"+/dl/;"32"+/u/<2pt>
	\ar@{-}"22"+/dr/;"33"+/u/<-2pt>
	\ar@{-}"33"+/dr/;"44"+/u/<-2pt>
	\end{xy}.
\]
The endomorphism algebra $\underline{\End}_{\Pi(w)}^{\mathbb{Z}}(M)$ is given by the following quiver with relations
\[
\Delta =\left[
\begin{xy}
	(0,0)+<0cm,0cm>="O",
	"O"+<0cm,0.8cm>="3"*{3},
	"O"+<-0.8cm,0cm>="2"*{2},
	"O"+<-1.6cm,-0.8cm>="1"*{1},
	"1"+<2.4cm,0cm>="4"*{4},
	
	\ar^{b}"1"+/ur/;"2"+/dl/
	\ar^{c}@(u,l)"1"+/u/;"3"+/l/
	\ar^{a}"4"+/l/;"1"+/r/
	\end{xy}
\right], \quad ab=ac=0.
\]
It is easy to see that the algebra $\underline{\End}_{\Pi(w)}^{\mathbb{Z}}(M)$ is derived equivalent to the path algebra of Dynkin quiver of type $D_{4}$.

$(\rm c)$
Let $Q$ be a quiver
	$$
	\begin{xy} (0,0)="O",
	"O"+<0cm,0cm>="1"*{1},
	"1"+<1cm,0cm>="2"*{2},
	"2"+<1cm,0.5cm>="3"*{3},
	"2"+<1cm,-0.5cm>="4"*{4},
	
	\ar"1"+/r/;"2"+/l/<-1pt>
	\ar"1"+/r/;"2"+/l/<1pt>
	\ar"2"+/ur/;"3"+/l/<-1pt>
	\ar"2"+/ur/;"3"+/l/<1pt>
	\ar"2"+/dr/;"4"+/l/<-1pt>
	\ar"2"+/dr/;"4"+/l/<1pt>
	
	\end{xy}$$
and $w$ be an element of $W_{Q}$ with its reduced expression ${\mbf{w}}=s_{4}s_{1}s_{2}s_{3}s_{2}s_{3}s_{1}s_{2}s_{4}$.
An expression $s_{1}s_{2}s_{3}s_{4}$ is an expression of the Coxeter element of $W_{Q}$.
The expression ${\mbf{w}}$ contains $s_{1}s_{2}s_{3}s_{4}$ as a subword, and hence the global dimension of $\Pi(w)_{0}$ is at most one.
We can see that ${\mbf{w}}$ satisfies $(\diamondsuit)$.
Thus $M=M(\mbf{w})$ is a tilting object of $\USub^{\mathbb{Z}}\Pi(w)$.
The endomorphism algebra $\underline{\End}_{\Pi(w)}^{\mathbb{Z}}(M)$ is given by the following quiver with relations:
\[
\Delta =\left[
\begin{xy}
	(0,0)="O",
	"O"+<0cm,0cm>="1"*{\bullet},
	"1"+<1cm,0cm>="2"*{\bullet},
	"2"+<1cm,0cm>="3"*{\bullet},
	"3"+<1cm,0cm>="4"*{\bullet},
	"4"+<1cm,0cm>="5"*{\bullet},
	
	\ar"2"+/r/;"3"+/l/<-1pt>
	\ar"2"+/r/;"3"+/l/<1pt>
	\ar^{a}"3"+/r/;"4"+/l/
	\ar^{b}"4"+/r/;"5"+/l/<1pt>
	\ar_{c}"4"+/r/;"5"+/l/<-1pt>
	\end{xy}
\right], \quad ab=ac=0.
\]
Note that ${\mbf{w}}$ is neither $c$-ending on $Q_{0}$ nor $c$-starting on $Q_{0}$.
\end{example}
There exist examples such that a reduced expression ${\mbf{w}}$ does not satisfies $(\diamondsuit)$, but $M=M(\mbf{w})$ is a titling object.
In fact, in the following example, $\Hom_{\Pi(w)}^{\mathbb{Z}}(M,\Omega(M))\neq 0$, but $\underline{\Hom}_{\Pi(w)}^{\mathbb{Z}}(M,M[-1])=0$ holds.
\begin{example}\label{exnotdia}
Let $Q$ be the same quiver as in Example \ref{extilt} (a) and $w$ be an element of $W_{Q}$ with its reduced expression ${\mbf{w}}=s_{3}s_{1}s_{2}s_{3}s_{1}s_{3}$.
Note that ${\mbf{w}}$ does not satisfies $(\diamondsuit)$.
We have
$$
M^{1}={\bf 3},
\quad
M^{2}=\begin{xy} (0,0)="O",
	"O"+<0.3cm,0.25cm>="11"*{{\bf 1}},
	"11"+<0.3cm,-0.5cm>="22"*{3},
	\end{xy},
\quad
	M^{3}=\Pi(w)e_{2}=\begin{xy} (0,0)="O",
	"O"+<0cm,0.5cm>="11"*{{\bf 2}},
	"11"+<-0.3cm,-0.5cm>="21"*{3},
	"21"+<-0.3cm,-0.5cm>="31",
	"31"+<-0.3cm,-0.5cm>="41",
	"41"+<-0.3cm,-0.5cm>="51",
	"21"+/r0.6cm/="22"*{{\bf 1}},
	"31"+/r0.6cm/="32",
	"32"+/r0.6cm/="33"*{3},
	
	\ar@{-}"11"+/dr/;"22"+/u/<-2pt>
	\end{xy},
$$
$$
	M^{4}=\begin{xy} (0,0)="O",
	"O"+<0cm,0.6cm>="11"*{{\bf 3}},
	"11"+<-0.3cm,-0.5cm>="21"*{\bf 1},
	"21"+<-0.3cm,-0.5cm>="31",
	"31"+<-0.3cm,-0.5cm>="41",
	"21"+/r0.6cm/="22"*{\bf 2},
	"31"+/r0.6cm/="32"*{3},
	"32"+/r0.6cm/="33"*{\bf 1},
	"41"+/r0.6cm/="42",
	"42"+/r0.6cm/="43",
	"43"+/r0.6cm/="44"*{3},
	\ar@{-}"11"+/dl/;"21"+/u/<2pt>
	\ar@{-}"11"+/dr/;"22"+/u/<-2pt>
	\ar@{-}"22"+/dr/;"33"+/u/<-2pt>
	\end{xy},
\quad
	M^{5}=\Pi(w)e_{1}=\begin{xy} (0,0)="O",
	"O"+<0cm,1cm>="11"*{{\bf 1}},
	"11"+<-0.3cm,-0.5cm>="21"*{2},
	"21"+<-0.3cm,-0.5cm>="31"*{3},
	"31"+<-0.3cm,-0.5cm>="41",
	"41"+<-0.3cm,-0.5cm>="51",
	"21"+/r0.6cm/="22"*{3},
	"31"+/r0.6cm/="32"*{1},
	"32"+/r0.6cm/="33"*{2},
	"41"+/r0.6cm/="42",
	"42"+/r0.6cm/="43"*{3},
	"43"+/r0.6cm/="44"*{1},
	"51"+/r0.6cm/="52",
	"52"+/r0.6cm/="53",
	"53"+/r0.6cm/="54",
	"54"+/r0.6cm/="55"*{3},
	
	\ar@{-}"21"+/dr/;"32"+/u/<-2pt>
	\ar@{-}"22"+/dl/;"32"+/u/<2pt>
	\ar@{-}"22"+/dr/;"33"+/u/<-2pt>
	\ar@{-}"33"+/dr/;"44"+/u/<-2pt>
	\end{xy},
\quad
	M^{6}=\Pi(w)e_{3}=\begin{xy} (0,0)="O",
	"O"+<0cm,0.6cm>="11"*{{\bf 3}},
	"11"+<-0.3cm,-0.5cm>="21"*{\bf 1},
	"21"+<-0.3cm,-0.5cm>="31"*{2},
	"31"+<-0.3cm,-0.5cm>="41"*{3},
	"21"+/r0.6cm/="22"*{\bf 2},
	"31"+/r0.6cm/="32"*{3},
	"32"+/r0.6cm/="33"*{\bf 1},
	"41"+/r0.6cm/="42",
	"42"+/r0.6cm/="43",
	"43"+/r0.6cm/="44"*{3},
	\ar@{-}"11"+/dl/;"21"+/u/<2pt>
	\ar@{-}"11"+/dr/;"22"+/u/<-2pt>
	\ar@{-}"22"+/dr/;"33"+/u/<-2pt>
	\end{xy}.
$$
It is easy to see that $\Hom_{\Pi(w)}^{\mathbb{Z}}(M^{2},\Omega(M^{1}))\neq0$ and $\underline{\Hom}_{\Pi(w)}^{\mathbb{Z}}(M^{2},\Omega(M^{1}))=0$.
Moreover, we see that $\underline{\Hom}_{\Pi(w)}^{\mathbb{Z}}(M,M[-1])=0$.
The expression $\mbf{w}$ contains an expression of the Coxeter element of $W_{Q}$.
Therefore, $M=M(\mbf{w})$ is a tilting object of $\USub^{\mathbb{Z}}\Pi(w)$.
\end{example}
\section{The relationship with the result of Amiot-Reiten-Todorov}\label{sectionrelation}
Before describing the result of \cite{ART}, we recall the definition of cluster categories which are introduced by Amiot \cite{Amiot09}.
Let $A$ be a finite dimensional algebra of global dimension at most two.
We denote by $\mathbb{S}=-\otimes_{A}^{\bf L}\kD A$ a Serre functor on ${\mathsf D}^{{\rm b}}(\mod A)$.
Put $\mathbb{S}_{2}=\mathbb{S}\circ[-2]$.
A {\it cluster category} $\mathsf{C}(A)$ of $A$ is the triangulated hull of  the orbit category ${\mathsf D}^{{\rm b}}(\mod A)/\mathbb{S}_{2}$ in the sense of Keller \cite{Ke05}.
We have the composition of triangle functors
\begin{align*}
\pi_{A}: {\mathsf D}^{{\rm b}}(\mod A) \to {\mathsf D}^{{\rm b}}(\mod A)/\mathbb{S}_{2} \to \mathsf{C}(A).
\end{align*}

Let $w\in W_{Q}$.
For a reduced expression ${\mbf{w}}=s_{u_{1}}s_{u_{2}}\cdots s_{u_{l}}$ of $w$,
let 
\begin{align*}
&M(\mbf{w})^{i}=M^{i}=\left(\Pi/I(s_{u_{1}}s_{u_{2}}\cdots s_{u_{i}})\right)e_{u_{i}}, \quad M(\mbf{w})=M=\bigoplus_{i=1}^{l}M(\mbf{w})^{i},\\
&A(\mbf{w})=A=\End_{\Pi(w)}^{\mathbb{Z}}(M(\mbf{w})).
\end{align*}
We denote by $e_{i}$ the idempotent of $A$ associated with $M^{i}$ for each $1\leq i \leq l$.
Let $e_{F}=\sum_{j\in F}e_{j}$, where $F=\{p_{u} \mid u\in\supp(w) \}$.
Put
\begin{align*}
\underline{A}=A/Ae_{F}A.
\end{align*}
By definition, we have an exact sequence
\begin{align}\label{AtoAeA}
0 \to Ae_{F}A \to A \to \underline{A} \to 0.
\end{align}
Note that, by the definition, $M$ is a right $A$-module and we have $Me_{F}=\Pi(w)$ as left $\Pi(w)$-modules.

We see that the algebra $\underline{A}$ coincides with the our endomorphism algebra $\underline{\End}_{\Pi(w)}^{\mathbb{Z}}(M)$.
\begin{lemma}\label{projfac}
We have $Ae_{F} A=\mathcal{P}(M,M)$.
In particular, we have $\underline{A}=\underline{\End}_{\Pi(w)}^{\mathbb{Z}}(M)$.
\end{lemma}
\begin{proof}
Clearly $Ae_{F} A\subset\mathcal{P}(M,M)$ holds.
Let $f \in \mathcal{P}(M,M)$.
We can assume that $f$ factors through $(\Pi(w))(j)=Me_{F}(j)$ for some $j\in\mathbb{Z}$.
Then we have a morphism $g:M \to Me_{F}$ of degree $j$ and $h : Me_{F} \to M$ of degree $-j$ such that $f=gh$.
Since $\End_{\Pi(w)}(M)$ is positively graded by Lemma \ref{gradedhom}, we have $j=0$.
This means $f\in Ae_{F} A$.
\end{proof}
Next we recall the result of \cite{ART}.
We denote by $\rho_{\Pi(w)}$ the composite of triangle functors
\begin{align*}
\rho_{\Pi(w)}: {\mathsf D}^{{\rm b}}(\mod\,\Pi(w)) \to {\mathsf D}^{{\rm b}}(\mod\,\Pi(w))/{\mathsf K}^{{\rm b}}(\proj\,\Pi(w)) \xto{\sim} \underline{\Sub}\Pi(w).
\end{align*}
Amiot-Reiten-Todorov showed the following theorem.
\begin{theorem}\cite[Theorem 3.1, Theorem 4.4]{ART}\label{ART}
Let $w\in W_{Q}$ and ${\mbf{w}}$ be a reduced expression of $w$.
Put $N:=M\otimes_{A}^{\bf L}\underline{A} \in {\mathsf D}^{{\rm b}}(\mod(\Pi(w)\otimes\underline{A}^{\rm op} ))$.
If ${\mbf{w}}$ is $c$-ending on $\supp(w)$, then we have the following.
\begin{itemize}
\item[(a)] The global dimension of $\underline{A}$ is at most two.

\item[(b)]
There exists a triangle equivalence $G : \mathsf{C}(\underline{A}) \to \USub\Pi(w)$ which makes the following diagram commutative up to isomorphism of functors

$$\xymatrix{
{\mathsf D}^{{\rm b}}(\mod\underline{A}) \ar^{\pi_{\underline{A}}}[d] \ar[rr] ^{N\otimes_{\underline{A}}^{\bf L}-} && {\mathsf D}^{{\rm b}}( \mod\Pi(w)) \ar^{\rho_{\Pi(w)}}[d] \\
\mathsf{C}(\underline{A}) \ar@{-->}[rr]^{G} && \USub\Pi(w).
}$$
\end{itemize}
\end{theorem}
\begin{remark}
For any reduced expression $\mbf{w}$ of $w \in W_{Q}$, since $Q$ is acyclic, there exists a quiver $Q^{\prime}$ such that whose underlying graph coincides with that of $Q$ and $\mbf{w}$ is $c$-ending on $\supp(w)$ as an element of $W_{Q^{\prime}}$.
Since $\USub\Pi(w)$ is independent of an orientation of $Q$, we have an equivalence (\ref{aireq}) by Theorem \ref{ART}.
\end{remark}
We construct a functor $\Phi:{\mathsf D}^{{\rm b}}(\mod\,\underline{A}) \to \underline{\Sub}^{\mathbb{Z}}\Pi(w)$ as follows.
By Definition \ref{gradontensor}, the algebra $\Pi(w)\otimes A^{\rm op}$ is a graded algebra and $M$ is a graded $\Pi(w)\otimes A^{\rm op}$-module.
Therefore $N=M\otimes_{A}^{\bf L}\underline{A}$ is an object of ${\mathsf D}^{{\rm b}}(\mod^{\mathbb{Z}}(\Pi(w)\otimes\underline{A}^{\rm op} ))$ and we have a derived functor
\begin{align*}
N\otimes_{\underline{A}}^{\bf L}- :{\mathsf D}^{{\rm b}}(\mod\,\underline{A})\to{\mathsf D}^{{\rm b}}(\mod^{\mathbb{Z}}\Pi(w)).
\end{align*}
We denote by $\rho_{\Pi(w)}^{\mathbb{Z}}$ the graded version of $\rho_{\Pi(w)}$, that is, 
\begin{align*}
\rho_{\Pi(w)}^{\mathbb{Z}}: {\mathsf D}^{{\rm b}}(\mod^{\mathbb{Z}}\Pi(w)) \to {\mathsf D}^{{\rm b}}(\mod^{\mathbb{Z}}\Pi(w))/{\mathsf K}^{{\rm b}}(\proj^{\mathbb{Z}}\Pi(w)) \xto{\sim} \USub^{\mathbb{Z}}\Pi(w).
\end{align*}
By composing $N\otimes_{\underline{A}}^{\bf L}-$ and $\rho_{\Pi(w)}^{\mathbb{Z}}$, we have a triangle functor
\begin{align*}
\Phi=\rho_{\Pi(w)}^{\mathbb{Z}} \circ N\otimes_{\underline{A}}^{\bf L}- : {\mathsf D}^{{\rm b}}(\mod\,\underline{A}) \to \underline{\Sub}^{\mathbb{Z}}\Pi(w).
\end{align*}

In this section, we show the following theorem which is a graded version of Theorem \ref{ART}.
\begin{theorem}\label{mainthm}
Let $w\in W_{Q}$ and ${\mbf{w}}$ be a reduced expression of $w$.
If ${\mbf{w}}$ is $c$-ending on $\supp(w)$, then we have the following.
\begin{itemize}
\item[(a)]
The triangle functor $\Phi=\rho_{\Pi(w)}^{\mathbb{Z}} \circ N\otimes_{\underline{A}}^{\bf L}-:{\mathsf D}^{{\rm b}}(\mod\,\underline{A}) \to \underline{\Sub}^{\mathbb{Z}}\Pi(w)$ is an equivalence.
\item[(b)]
We have the following commutative diagram up to isomorphism of functors
$$\xymatrix{
{\mathsf D}^{{\rm b}}(\mod\,\underline{A}) \ar[rr]^{\Phi} \ar[d]^{\pi_{\underline{A}}} && \underline{\Sub}^{\mathbb{Z}}\Pi(w) \ar[d]^{{\rm Forget}} \\
\mathsf{C}(\underline{A}) \ar[rr]^{G} && \underline{\Sub}\,\Pi(w).
}$$
\end{itemize}
\end{theorem}
We begin with the following lemma.
\begin{lemma}\label{AeA}\cite[Lemma 3.2]{ART}
If a reduced expression ${\mbf{w}}=s_{u_{1}}s_{u_{2}}\cdots s_{u_{l}}$ of $w$ is $c$-ending on $\supp(w)$, then we have a projective resolution $0\to P^{1} \to P^{0} \to Ae_{i} \to \underline{A}e_{i} \to 0$ of $A$-module $\underline{A}e_{i}$, where $i\in\{1\leq j \leq l\}\setminus F$ and $P^{0},P^{1}\in \add(Ae_{F})$.
\end{lemma}
\begin{proof}
Since $\mbf{w}$ is $c$-ending on $\supp(w)$ and by \cite[Lemma 4.3]{ART}, the conditions $(H1) \sim (H4)$ in \cite{ART} are satisfied.
Then the assertion follows immediately from \cite[Lemma 3.2]{ART}.
\end{proof}
We need the following lemma.
\begin{lemma}\label{projresol}
If a reduced expression ${\mbf{w}}=s_{u_{1}}s_{u_{2}}\cdots s_{u_{l}}$ of $w$ is $c$-ending on $\supp(w)$, then we have the following.
\begin{itemize}
\item[(a)]
$Ae_{F}Ae_{i}=Ae_{i}$ for any $i\in F$.
\item[(b)]
We have a projective resolution of $Ae_{F}A$ as an $A$-module
\begin{align}\label{resolofAeA}
0 \to P^{1} \to P^{0} \to Ae_{F}A \to 0,
\end{align}
where $P^{0}, P^{1}\in\add(Ae_{F})$.
\item[(c)]
We have $M\otimes_{A}^{\bf L}(Ae_{F}A)\in{\mathsf K}^{{\rm b}}(\proj^{\mathbb{Z}} \Pi(w))$.
\end{itemize}
\end{lemma}
\begin{proof}
(\rm a) Since $e_{i}$ is an idempotent, this is clear.

(\rm b)
We have an exact sequence $(\ref{AtoAeA})$.
Thus the assertion follows from $(\rm a)$ and Lemma \ref{AeA}.

(\rm c)
By $(\rm b)$, $Ae_{F}A\in\thick Ae_{F}$ holds.
Thus we have $M\otimes_{A}^{\bf L}(Ae_{F}A)\in\thick(M\otimes_{A}^{\bf L}Ae_{F})={\mathsf K}^{{\rm b}}(\proj^{\mathbb{Z}} \Pi(w))$, where the last equality follows from $M\otimes_{A}^{\bf L}Ae_{F}=Me_{F}=\Pi(w)$.
\end{proof}
Then we are ready to show the main theorem.
\begin{proof}[Proof of Theorem \ref{mainthm}]
(\rm a)
We first show that $\Phi(\underline{A})=\rho_{\Pi(w)}^{\mathbb{Z}}(N\otimes_{\underline{A}}^{\bf L}\underline{A})\simeq M$ in $\underline{\Sub}^{\mathbb{Z}}\Pi(w)$.
Recall that $N:=M\otimes_{A}^{\bf L}\underline{A}$.
By applying $M\otimes_{A}^{\bf L}-$ to the sequence (\ref{AtoAeA}), we have the following triangle in ${\mathsf D}^{{\rm b}}(\mod^{\mathbb{Z}}\Pi(w))$
\begin{align*}
M\otimes_{A}^{\bf L}(Ae_{F}A) \to M\otimes_{A}^{\bf L}A \to M\otimes_{A}^{\bf L}\underline{A} \to M\otimes_{A}^{\bf L}(Ae_{F}A)[1].
\end{align*}
By Lemma \ref{projresol} (c) and this triangle, $M$ is isomorphic to $\rho_{\Pi(w)}^{\mathbb{Z}}(N\otimes_{\underline{A}}^{\bf L}\underline{A})$ in $\underline{\Sub}^{\mathbb{Z}}\Pi(w)$.

By Theorem \ref{tilting}, $M$ is a tilting object in $\underline{\Sub}^{\mathbb{Z}}\Pi(w)$.
Since the global dimension of $\underline{A}$ is at most two, $\underline{A}$ is a tilting object of ${\mathsf D}^{{\rm b}}(\mod\,\underline{A})$.
Therefore the functor $\rho_{\Pi(w)}^{\mathbb{Z}}\circ (N\otimes_{\underline{A}}^{\bf L}-)$ is an equivalence by Lemma \ref{trifunclem}.

(\rm b)
We have the following commutative diagram up to isomorphism of functors
$$
\begin{xy}
(0,0)="O",
"O"+<0cm,0cm>="barA"*{{\mathsf D}^{{\rm b}}(\mod\,\underline{A})},
"barA"+<8cm,0cm>="zpi"*{{\mathsf D}^{{\rm b}}(\mod^{\mathbb{Z}}\Pi(w))},
"barA"+<4cm,-1cm>="pi"*{{\mathsf D}^{{\rm b}}(\mod\,\Pi(w))},
"barA"+<0cm,-2cm>="cluster"*{\mathsf{C}(\underline{A})},
"zpi"+<0cm,-2cm>="subz"*{\underline{\Sub}^{\mathbb{Z}}\Pi(w)},
"pi"+<0cm,-2cm>="sub"*{\underline{\Sub}\,\Pi(w)},

\ar^{N\otimes_{\underline{A}}^{\bf L}-}"barA"+<1.1cm,0cm>;"zpi"+<-1.3cm,0cm>
\ar_{N\otimes_{\underline{A}}^{\bf L}-}"barA"+<1.1cm,-0.3cm>;"pi"+<-1.2cm,0.1cm>
\ar"zpi"+<-1.3cm,-0.3cm>;"pi"+<1.2cm,0.1cm>

\ar^{G}"cluster"+<0.5cm,-0.2cm>;"sub"+<-0.8cm,0cm>
\ar"subz"+<-0.8cm,-0.2cm>;"sub"+<0.7cm,0cm>

\ar^{\pi_{\underline{A}}}"barA"+<0cm,-0.3cm>;"cluster"+<0cm,0.3cm>
\ar^{\rho_{\Pi(w)}^{\mathbb{Z}}}"zpi"+<0cm,-0.3cm>;"subz"+<0cm,0.3cm>
\ar^{\rho_{\Pi(w)}}"pi"+<0cm,-0.3cm>;"sub"+<0cm,0.3cm>
\end{xy}
$$
where ${\mathsf D}^{{\rm b}}(\mod^{\mathbb{Z}}\Pi(w))\to{\mathsf D}^{{\rm b}}(\mod\Pi(w))$ and $\underline{\Sub}^{\mathbb{Z}}\Pi(w)\to \USub\Pi(w)$ are degree forgetful functors.
In particular, we obtain the desired diagram.
\end{proof}
\section*{Acknowledgements}
The author is supported by Grant-in-Aid for JSPS Fellowships 15J02465.
He would like to thank my supervisor Osamu Iyama for his support and many helpful comments.
He is grateful to Kota Yamaura for helpful comments and discussions.
Some results of this paper were given during the author's visit to Trondheim.
He thanks Idun Reiten for her many supports during his stay.

\end{document}